\documentclass[reqno]{amsart}

\usepackage[english]{babel}
\usepackage{amsmath}
\usepackage{amsthm}
\usepackage{amsfonts}
\usepackage{amssymb}
\usepackage{anysize}
\usepackage{hyperref}
\usepackage{appendix}
\usepackage{mathtools}
\usepackage{graphicx}
\usepackage{pgfplots}
\usepackage{subcaption} 
\usepackage{enumitem}
\usepackage{dsfont}

\makeatletter
\@namedef{subjclassname@2020}{\textup{2020} Mathematics Subject Classification}
\makeatother

\newtheorem{definition}{Definition}[section]
\newtheorem{lemma}[definition]{Lemma}
\newtheorem{theorem}[definition]{Theorem}
\newtheorem{corollary}[definition]{Corollary}
\newtheorem{proposition}[definition]{Proposition}

\newtheorem{remark}[definition]{Remark}

\DeclareMathOperator{\divop}{div}

\DeclareMathOperator{\supp}{supp}

\DeclareMathOperator{\sgn}{sgn}

\DeclareMathOperator{\esssup}{ess~sup}

\DeclareMathOperator{\blambda}{\boldsymbol{\lambda}}
\DeclareMathOperator{\bsigma}{\boldsymbol{\sigma}}
\DeclareMathOperator{\bmu}{\boldsymbol{\mu}}
\DeclareMathOperator{\bF}{\boldsymbol{F}}
\DeclareMathOperator{\bV}{\boldsymbol{V}}

\DeclareMathOperator{\alphalpha}{\beta}
\DeclareMathOperator{\betabeta}{\alpha}

\numberwithin{equation}{section}

\title[]{Uniform contractivity of the Fisher infinitesimal model\\ with strongly convex selection}

\author[Vincent Calvez]{Vincent Calvez}
\address{Institut Camille Jordan (ICJ), UMR 5208 CNRS \& Universit\'e Claude Bernard Lyon 1, 69100 Villeurbanne, France}
\email{vincent.calvez@math.cnrs.fr}

\author[David Poyato]{David Poyato}
\address{Departamento de Matem\'atica Aplicada and Research Unit ``Modeling Nature'' (MNat), Facultad de Ciencias, Universidad de Granada, 18071 Granada, Spain}
\email{davidpoyato@ugr.es}

\author[Filippo Santambrogio]{Filippo Santambrogio}
\address{Institut Camille Jordan (ICJ), UMR 5208 CNRS \& Universit\'e Claude Bernard Lyon 1, 69100 Villeurbanne, France}
\email{santambrogio@math.univ-lyon1.fr}

\begin{document}

\date{}


\thanks{\textbf{Acknowledgment.} The authors are indebted to Thomas Lepoutre for stimulating discussions and valuable suggestions on the paper's presentation. We also thank Laurent Lafleche for pointing us the identity in Corollary \ref{C-nonlinear-dualities}. VC and DP have received funding from the European Research Council (ERC) under the European Union’s Horizon 2020 research and innovation program (grant agreement No 865711). DP has received founding from the European Union’s Horizon Europe research and innovation program under the Marie Sk\l odowska-Curie grant agreement No 101064402, and partially from the MINECO-Feder (Spain) project P18-RT-2422. FS acknowledges the support of the Lagrange Mathematics and Computation Research Center project on Optimal Transportation.}

\begin{abstract}
The Fisher infinitesimal model is a classical model of phenotypic trait inheritance in quantitative genetics. Here, we prove that it encompasses a remarkable convexity structure which is compatible with a selection function having a convex shape. It yields uniform contractivity along the flow, as measured by a $L^\infty$ version of the Fisher information. It induces in turn asynchronous exponential growth of solutions, associated with a well-defined, log-concave, equilibrium distribution. Although the equation is non-linear and non-conservative, our result shares some similarities with the Bakry-Emery approach to the exponential convergence of solutions to the Fokker-Planck equation with a convex potential. Indeed, the contraction takes place at the level of the Fisher information. Moreover, the key lemma for proving contraction involves the Wasserstein distance $W_\infty$ between two probability distributions of a (dual) backward-in-time process, and it is inspired by a maximum principle by Caffarelli for the Monge-Amp\`ere equation.
\end{abstract}

\maketitle

\section{Introduction}\label{S-introduction}

Let us consider the following nonlinear model
\begin{equation}\label{E-nonlinear}
F_n=\mathcal{T}[F_{n-1}],\quad n\in \mathbb{N},\,x\in \mathbb{R},
\end{equation}
describing the evolution of the distribution $F_n=F_n(x)$ of a one-dimensional trait $x\in \mathbb{R}$, subject to sexual reproduction and the effect of selection at each generation. The operator $\mathcal{T}$ above is defined by
\begin{align}
\mathcal{T}[F](x)&:=e^{-m(x)}\mathcal{B}[F](x), \quad x\in \mathbb{R},\label{E-operator-T}\\
\mathcal{B}[F](x)&:=\iint_{\mathbb{R}^2} G\left(x-\frac{x_1+x_2}{2}\right)\,F(x_1)\,\frac{F(x_2)}{\Vert F\Vert_{L^1}}\,dx_1\,dx_2, \quad x\in\mathbb{R},\label{E-operator-B}
\end{align}
for any $F\in L^1_+(\mathbb{R})\setminus\{0\}$. On the one hand, the operator $\mathcal{B}$ describes the distribution of traits of descendants of the previous generation $F_{n-1}$, arising as recombination of parental traits in agreement with {\it Fisher's infinitesimal model}, which is a classical model in quantitative genetics \cite{BEV-17,F-18}. Accordingly, the mixing kernel $G$ is set to a centered Gaussian distribution with unit {\it segregation variance} without loss of generality, namely
\begin{equation}\label{E-G}
G(x):=\frac{1}{(2\pi)^{1/2}}e^{-\frac{x^2}{2}},\quad x\in \mathbb{R}.
\end{equation}
On the other hand, the trait-dependent mortality function $m=m(x)\geq 0$ represents the effect of selection on the population, which acts multiplicatively over the descendants. In other words, the multiplicative factor $e^{-m(x)}$ in \eqref{E-operator-T} represents the survival probability to the next generation of individuals having the trait $x$. We note that the time-discrete generations $n\in \mathbb{N}$ are assumed non-overlapping since, altogether, $F_n$ describes the distribution of those offspring of $F_{n-1}$ having survived after the selection step, and then different generations do not get mixed, see \cite{CLP-21-arxiv} for further insight.

The goal of this paper is to extend the studies initiated in \cite{CLP-21-arxiv} to a broader class of selection functions. Specifically, when $m$ is a strongly convex function we prove {\it asynchronous exponential growth} in \eqref{E-nonlinear}. In other words, we derive quantitative rates for the relaxation of the solutions $\{F_n\}_{n\in \mathbb{N}}$ of \eqref{E-nonlinear} to a strongly log-concave quasi-equilibrium of the form $\blambda^n \bF$, where $\blambda>0$ and $\bF\in L^1(\mathbb{R})\cap \mathcal{P}(\mathbb{R})$ is an appropriate probability density. We remark that in order for an anstaz $F_n(x)=\lambda^n F(x)$ to define a solution to \eqref{E-nonlinear}, we need that the pair $(\lambda,F)$ solves the following nonlinear eigenproblem:
\begin{align}\label{E-nonlinear-eigenproblem}
\begin{aligned}
&\lambda F=\mathcal{T}[F],\quad x\in \mathbb{R},\\
& F\geq 0,\quad \int_{\mathbb{R}} F(x)\,dx=1.
\end{aligned}
\end{align}
Hence, the possible quasi-equilibria are to be found as solutions to \eqref{E-nonlinear-eigenproblem}. Note that contrarily to the special quadratic regime treated in \cite{CLP-21-arxiv}, the Gaussian structure cannot be longer exploited and, in particular, the existence of solutions to \eqref{E-nonlinear-eigenproblem} is unclear. Indeed, the above non-linear integral operator is $1$-homogenenous but non-monotone, and therefore the Krein-Ruthman theorem \cite{M-07} cannot be applied as it has been done in other (usually linear) problems in population dynamics \cite{BCV-16,CLW-17}. Hence, the study of the non-linear evolution problem \eqref{E-nonlinear} and the non-linear eigenproblem \eqref{E-nonlinear-eigenproblem} requires innovative ideas.

Along this paper, we address jointly the following two problems: (i) Existence of a strongly log-concave solution $(\blambda,\bF)$ to \eqref{E-nonlinear-eigenproblem}, and (ii) Quantitative relaxation of the solutions to \eqref{E-nonlinear} towards the quasi-equilibrium $\blambda^n \bF$. 
We make the crucial hypothesis that $m$ is a strongly convex function,
\begin{equation}\label{H-main-m-1}\tag{H1}
m'' \geq \betabeta \quad \mbox{for some}\quad \betabeta>0,
\end{equation}
The function $m$ necessarily reaches its minimum value over $\mathbb{R}$. For convenience, we assume the following additional hypothesis without loss of generality, 
\begin{equation}\label{H-main-m-2}\tag{H2}
m \geq 0,\quad \mbox{and}\quad m(0)=0.
%&m(x)=m(-x), \quad x\in \mathbb{R},\label{H-main-m-2}\tag{H2}
\end{equation}
The $L^\infty$ relative Fisher information $\mathcal{I}_\infty$ plays a pivotal role in our analysis, as it measures the contractivity along the flow (see methodological notes below). It is defined as follows, for a pair of functions $P,Q\in L^1_+(\mathbb{R})\cap C^1(\mathbb{R})$,
\begin{equation}\label{E-main-Fisher-information-infty}
\mathcal{I}_\infty(P\Vert Q):=\left\Vert \dfrac{d}{dx}\left(\log \frac{P}{Q}\right)\right\Vert_{L^\infty}.
\end{equation}

\begin{theorem}\label{T-main}
Assume \eqref{H-main-m-1}-\eqref{H-main-m-2}.
Then, the following statements hold true:
\begin{enumerate}[label=(\roman*)]
\item {\rm{(\bf Existence of quasi-equilibrium})}\\
There is at least one solution $(\blambda,\bF)$ to \eqref{E-nonlinear-eigenproblem}. In addition, 
$\bF=e^{-\bV}\in L^1_+(\mathbb{R})\cap C^\infty(\mathbb{R})$ is $\alphalpha$-log-concave, where $\alphalpha>\frac12$ is uniquely defined by the following relationship
\begin{equation}\label{E-main-alpha}
\alphalpha=\betabeta+\frac{2\alphalpha}{1+2\alphalpha}.
%\quad \mbox{i.e.},\quad \alphalpha=\frac{(1+2\betabeta)+\sqrt{(1+2\betabeta)^2+8\betabeta}}{4}.
\end{equation}
Moreover, $(\blambda,\bF)$ is the unique solution to \eqref{E-nonlinear-eigenproblem} among all pairs $(\lambda,F)$ such that
\begin{equation}\label{E-main-F-uniqueness}
\dfrac{d}{dx}\left(\log \frac{F}{\bF}\right)\in L^\infty(\mathbb{R}).
\end{equation}
\item {\rm ({\bf One-step contraction})}\\
Consider any $F_0\in L^1_+(\mathbb{R})\cap C^1(\mathbb{R})$ such that 
\begin{equation}\label{H-main-F0}\tag{H3}
\dfrac{d}{dx} \left(\log \frac{F_0}{\bF}\right)\in L^\infty(\mathbb{R}),
\end{equation}
and let $\{F_n\}_{n\in \mathbb{N}}$ be the solution to \eqref{E-nonlinear} issued at $F_0$. Then, we have
\begin{equation}\label{E-main-one-step-contraction}
\mathcal{I}_\infty\left(\left. F_n \right\Vert \bF\right)\leq \frac{2}{1+2\alphalpha}\,\mathcal{I}_\infty\left(\left. F_{n-1} \right\Vert \bF\right),
\end{equation}
for any $n\in \mathbb{N}$. 
\item {\rm ({\bf Asynchronous exponential growth})}\\
Consider any $F_0\in L^1_+(\mathbb{R})\cap C^1(\mathbb{R})$ verifying the assumption \eqref{H-main-F0} above, and let $\{F_n\}_{n\in \mathbb{N}}$ be the solution to \eqref{E-nonlinear} issued at $F_0$. Then, we have
\begin{align}\label{E-main-asynchronous-exponential-growth}
\left\vert\frac{\Vert F_n\Vert_{L^1}}{\Vert F_{n-1}\Vert_{L^1}}-\blambda\right\vert&\leq C\, \left(\frac{2}{1+2\alphalpha}\right)^n,\\
\mathcal{D}_{\rm KL}\left(\left.\frac{F_n}{\Vert F_n\Vert_{L^1}}\right\Vert\bF\right)&\leq C\, \left(\frac{2}{1+2\alphalpha}\right)^{2n},
\end{align}
for every $n\in \mathbb{N}$, where $C>0$ is a explicit constant depending on $F_0$, and $\mathcal{D}_{\rm KL}$ is the Kullback-Leibler divergence (or relative entropy), that is,
\begin{equation}\label{E-main-relative-entropy}
\mathcal{D}_{\rm KL}(P\Vert Q):=\int_{\mathbb{R}}\log\left (\frac{P(x)}{Q(x)}\right )P(x)\,dx,\quad P,Q \in L^1_+(\mathbb{R})\cap \mathcal{P}(\mathbb{R}).
\end{equation}
\end{enumerate}
\end{theorem}

\begin{remark}[Case of quadratic selection]\label{R-quadratic-selection}
For quadratic selection $m(x)=\frac{\betabeta}{2}|x|^2$, we have that $m$ satisfies the hypothesis \eqref{H-main-m-1}-\eqref{H-main-m-2} in Theorem \ref{T-main}, and then our new result applies. Such a special case was studied in detail in \cite{CLP-21-arxiv}, where in particular it was proven that there is a unique eigenpair $(\blambda,\bF)$ of \eqref{E-nonlinear-eigenproblem}, which involves a Gaussian eigenfunction $\bF(x)=G_{0,\sigma^2}(x)$ with variance $\sigma^2>0$ verifying
\begin{equation}\label{E-quadratic-case-variance}
\frac{1}{\sigma^2}=\betabeta+\frac{1}{1+\frac{\sigma^2}{2}}.
\end{equation}
In particular, $\bF$ is $\frac{1}{\sigma^2}$-log-concave, which is compatible with our new result in view of the identity $\sigma^2=\alphalpha^{-1}$ stemming from equations \eqref{E-main-alpha} and \eqref{E-quadratic-case-variance}. Furthermore, the contraction factor in \eqref{E-main-one-step-contraction} predicted by Theorem \ref{T-main} also recovers the one obtained in \cite{CLP-21-arxiv} for quadratic selection. Specifically, one has
$$\frac{2}{1+2\alphalpha}=\frac{(3+2\betabeta)-\sqrt{(3+2\betabeta)^2-8}}{2},$$
which agrees precisely with the contraction factor found in \cite[Lemma 6.3]{CLP-21-arxiv}.
\end{remark}

\begin{remark}[Close-to-equilibrium initial data]
In contrast with \cite{CLP-21-arxiv}, where the above framework was restricted to $m(x)=\frac{\betabeta}{2}|x|^2$ but generic $F_0\in \mathcal{M}_+(\mathbb{R})$, Theorem \ref{T-main} applies to a broader class of selection functions verifying \eqref{H-main-m-1}-\eqref{H-main-m-2} at the cost of restricting to initial data fulfilling the hypothesis \eqref{H-main-F0}. Specifically, such a condition imposes a precise behavior of the tails of $F_0$, which must be very close to those of the eigenfunction $\bF$ (in particular, two Gaussian initial distributions should have the same variance). 
\end{remark}

\begin{remark}[Conditional uniqueness]
Another difference with \cite{CLP-21-arxiv} is that the current approach does not guarantee global uniqueness of solutions to the eigenproblem \eqref{E-nonlinear-eigenproblem}, but only within the class of eigenpairs verifying \eqref{E-main-F-uniqueness}. Nevertheless, we conjecture that global uniqueness holds true, as in the quadratic case $m(x)=\frac{\betabeta}{2}|x|^2$. Proving global uniqueness would require a careful control of the behavior at infinity, in the spirit of \cite{CLP-21-arxiv}, which is beyond the scope of this paper.
\end{remark}

\begin{remark}[Log-concavity and contraction factor]\label{R-parameters-behavior}
For any $\betabeta>0$, we have that the log-concavity parameter $\alphalpha$ in \eqref{E-main-alpha} and the corresponding contraction factor $\frac{2}{1+2\alphalpha}$ in \eqref{E-main-one-step-contraction} satisfy the following properties:
\begin{align*}
\betabeta\searrow 0 \quad &\Longrightarrow \quad \alphalpha\searrow \frac{1}{2} \quad \mbox{and}\quad \frac{2}{1+2\alphalpha}\nearrow 1,\\
\betabeta\nearrow \infty \quad &\Longrightarrow \quad \alphalpha\nearrow \infty \quad \mbox{and}\quad \frac{2}{1+2\alphalpha}\searrow 0,
\end{align*}
see Figure \ref{fig:alpha-contraction-factor}. In particular, we have genuine contraction in \eqref{E-main-one-step-contraction} since $0<\frac{2}{1+2\alphalpha}<1$ for every $\betabeta>0$.
\end{remark}

\begin{remark}[One-dimensional traits]
In this paper we restrict to one-dimensional traits, but note that an analogous version of \eqref{E-nonlinear} and \eqref{E-nonlinear-eigenproblem} makes sense in higher dimensions yet. In fact, these were studied in \cite{CLP-21-arxiv} for quadratic selection functions. However, a higher-dimensional version of our result for generic strongly convex selection function would require some non-trivial improvements of the present methods. Just to emphasize some non-trivial obstructions, we remark that our approach exploits a maximum principle for the Monge-Ampère equation in convex but not uniformly-convex domains, as described below. In this setting, it is not even clear why the standard elliptic regularity should hold up to the boundary, as in the seminal works \cite{C-96}. In two-dimensional domains with special symmetries, this theory has been developed recently in \cite{J-19}, but a higher dimensional extension would require further work which goes beyond the scope of this paper. 
\end{remark}

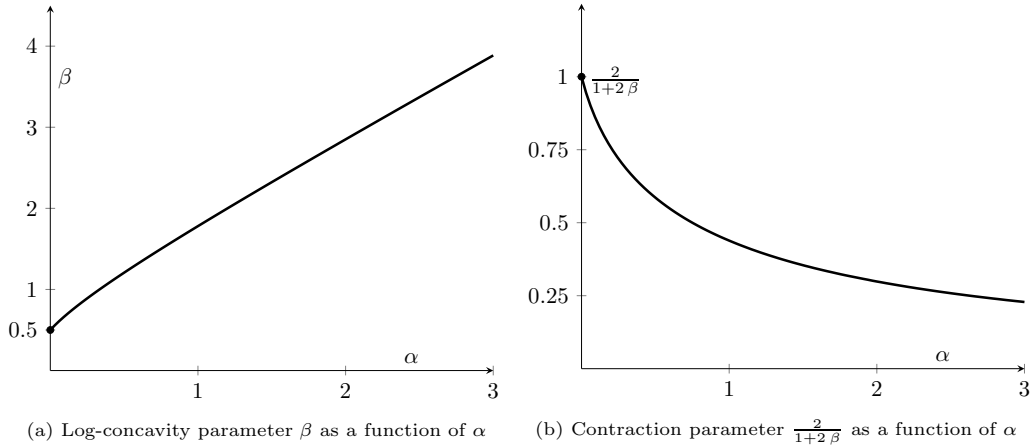
\begin{figure}
\begin{subfigure}{0.42\textwidth}
\begin{tikzpicture}[scale=0.85]
\begin{axis}[
  axis x line=middle, axis y line=middle,
  label/.style={at={(current axis.right of origin)},anchor=west},
  xmin=0, xmax=3, xtick={0,1,2,3}, xlabel=$\betabeta$,
  ymin=0, ymax=4.5, ytick={0,0.5,1,2,3,4}, ylabel=$\alphalpha$,
]
\addplot [
    domain=0:3, 
    samples=100, 
    color=black,
    line width=0.4mm,
]{((1+2*x+pow((1+2*x)^2+8*x,0.5)))/4};
\addplot [mark=*, mark size=1.5pt] coordinates {(0,0.5)};
\end{axis}
\end{tikzpicture}
\caption{Log-concavity parameter $\alphalpha$ as a function of $\betabeta$}
\label{subfig:alpha}
\end{subfigure}
\begin{subfigure}{0.42\textwidth}
\begin{tikzpicture}[scale=0.85]
\begin{axis}[
  axis x line=middle, axis y line=middle,
  label/.style={at={(current axis.right of origin)},anchor=west},
  xmin=0, xmax=3, xtick={0,1,2,3}, xlabel=$\betabeta$,
  ymin=0, ymax=1.25, ytick={0,0.25,0.5,0.75,1}, ylabel=$\frac{2}{1+2\alphalpha}$,
]
\addplot [
    domain=0:3, 
    samples=100, 
    color=black,
    line width=0.4mm,
]{((3+2*x)-pow((3+2*x)^2-8,0.5))/2};
\addplot [mark=*, mark size=1.5pt] coordinates {(0,1)};
\end{axis}
\end{tikzpicture}
\caption{Contraction parameter $\frac{2}{1+2\alphalpha}$ as a function of $\betabeta$}
\label{subfig:contraction-factor}
\end{subfigure}
\caption{Plot of the log-concavity parameter $\alphalpha$ of the eigenfunction $\bF$ and the contraction parameter $\frac{2}{1+2\alphalpha}$ in Theorem \ref{T-main} as a function of $\betabeta$}
\label{fig:alpha-contraction-factor}
\end{figure}

\medskip

\noindent {\bf Bibliographical notes}.

\smallskip

This work can be viewed as another brick to combine optimal transportation tools for non-conservative problems arising in biology. The connection between the Fisher infinitesimal model and the $L^2$ Wasserstein distance was spotted by {\sc G. Raoul} \cite{R-17-arxiv} (see also \cite{MR-13} for similar results in a different context of protein exchanges between cells). In fact, when there is no selection (that is, $m\equiv const$), the operator $\mathcal{T}$ is non-expansive for the latter distance. Contraction cannot be expected because of translational invariance. Nevertheless, it is contractive with rate $1/\sqrt{2}$ in the class of distributions having the same center of mass (the latter being preserved by the flow) \cite[Theorem 4.1 and Corollary 4.2]{R-17-arxiv}. This remarkable structure was exploited by {\sc G. Raoul} in a perturbative setting, when selection is small (in amplitude), and restricted to a compact interval ($m$ is constant beyond a certain range). More precisely, {\sc G. Raoul} proved that the dynamics is well captured by some averaged quantities (``moments'') of the Gaussian distribution coupled with the selection function, provided that the initial data is well-prepared, in the basin of attraction of the stationary state, and the amplitude of selection is small enough. For that purpose, he carefully established that the contraction issued from the infinitesimal operator was robust enough to dominate detrimental effects due to selection. Note that the later references consider overlapping generations, that is, a continuous-in-time rather than discrete dynamics. However, some fruitful analogy can be drawn between the results and methodology.    

In parallel, the regime of small segregation variance (when $G$ \eqref{E-G} has variance $\varepsilon^2$ and $\varepsilon$ is small enough) was investigated by \cite{CGP-19,P-20-arxiv} in another perturbative setting, without exploiting the Wasserstein metric structure. This methodology built upon the seminal works on vanishing viscosity limits associated with linear (asexual) modes of reproduction in quantitative genetics models \cite{DJMP-05,PB-08,BMP-09}. Interestingly, it was proven in \cite{CGP-19} that the problem \eqref{E-nonlinear-eigenproblem} lacks uniqueness in full generality. More precisely, it was possible to build a solution to \eqref{E-nonlinear-eigenproblem} centered in the vicinity of any local minimum of $m$, provided that the selection value at the local minimum is close enough to the global minimum. This result entails a clear separation with linear, order-preserving operators (and non-linear extensions \cite{M-07,N-88}) for which \eqref{E-nonlinear-eigenproblem} genuinely admits a unique solution (under standard irreducibility assumptions). 

Heuristically, uniqueness of the (non-linear) eigenpair $(\blambda,\bF)$ is rather clear when the selection function $m$ is convex, and \cite{CLP-21-arxiv} was a first contribution in this direction, restricted to $m(x)=\frac{\betabeta}{2}|x|^2$. By exploiting the quadratic structure of the operator $\mathcal{T}$ in \eqref{E-operator-T} (which involves products and convolutions by Gaussian density functions), it was possible to prove asynchronous exponential growth towards the explicit Gaussian distribution of equilibrium $\bF$, starting from any initial configuration $F_0$. This was achieved by a careful study of the binary tree of ancestors, together with explicit change of variables in a high-dimensional integral, to prove a sort of concentration of measure estimates. More precisely, it was shown that the traits of the ancestors decorrelate sufficiently fast, backward in the tree, from the trait of the individual at generation $n$. This implies that the dependence of the trait distribution $F_n$ at generation $n$ upon the initial distribution $F_0$ diminishes exponentially fast. Asynchronous exponential growth is a consequence of this observation, which is a backward feature.   

Last, but not least, let us mention that both the infinitesimal model \eqref{E-operator-T}, and the relative information \eqref{E-main-Fisher-information-infty} (or rather \eqref{E-main-Fisher-information-2} below) date back to a couple of seminal works by {\sc R. A. Fisher} in the same years (circa 1920) on seemingly different purposes, respectively \cite{F-18} and \cite{F-22}, see \cite{S-05} for a discussion. 

\medskip

\noindent {\bf Methodological notes}.

\smallskip

In the present study, we push further the observations of \cite{CLP-21-arxiv}. We identify a key mechanism ensuring a one-step contraction for the flow \eqref{E-nonlinear}. This can be summarized roughly as follows: 
\begin{quote} \em
For any two given individuals with traits $X$ and $X'$ respectively, the associated parental traits $(X_1,X_2)$ and $(X_1',X_2')$ are closer to each other than $(X,X')$ in some sense,
\end{quote} 
see also \cite[Appendix F.2]{GCBBLRC-22-arxiv} for a visual explanation.
To make sense of this contraction, we shall work with the $L^\infty$ Wasserstein distance, denoted by $W_\infty$ (in contrast with the usual $L^2$ Wasserstein distance). This naturally leads to estimates on the so-called $L^\infty$ relative Fisher information $\mathcal{I}_\infty$ \eqref{E-main-Fisher-information-infty} (in contrast to the usual ($L^2$) relative entropy $\mathcal{I}_2$, see \eqref{E-main-Fisher-information-2} below). The core estimate \eqref{E-main-one-step-contraction} is forward in time, and it naturally arises as a dual estimate of a backward in time estimate analogous to the work in \cite{CLP-21-arxiv}. 
\medskip

\paragraph*{\em A forward-backward argument.}
We propose a short warm-up to this argument, which may help the reader follow our method (without details of the proofs). Indeed, one complication of our setting is that each individual has two parents, so that the dimension of the distribution doubles at each generation. Nonetheless, the same methodology can be applied to the case of a single parent, which boils down to a {\em linear operator}. We thus consider, temporarily, the following linear operator:
\begin{equation}\label{E-operator-A}
\mathcal{A}[F](x):=e^{-m(x)}\int_{\mathbb{R}} G\left(x-y\right)\,F(y)\,dy,\quad x\in \mathbb{R},
\end{equation}
in place of the above non-linear operator $\mathcal{T}$ in \eqref{E-operator-T}. In this simpler case, the Krein-Rutman theorem can be applied (at least formally), and there exists an eigenpair $(\blambda,\bF)$ of the linear eigenproblem \eqref{E-nonlinear-eigenproblem} with $\mathcal{T}$ replaced by $\mathcal{A}$. Now, consider any solution $\{F_n\}_{n\in \mathbb{N}}$ to the time-discrete problem \eqref{E-nonlinear} with $\mathcal{T}$ replaced again by the linear operator $\mathcal{A}$. We may introduce the associated relative distribution $u_n = \frac{F_n}{\blambda^n \bF}$ to follow the trend of $F_n$ across generations. It satisfies the following equation:
$$
u_{n}(x) = \dfrac{\int_{\mathbb{R}} G(x-y)\,u_{n-1}(y)\bF(y)\, dy}{\int_{\mathbb{R}} G(x-z) \bF(z)\, dz} = \int_{\mathbb{R}} P(x;y)\,u_{n-1}(y) \, dy,\quad n\in \mathbb{N},\,x\in \mathbb{R},
$$
where the $x$-dependent probability distribution function $P(x;\cdot)$ is defined as 
\begin{equation}
P(x;y) = \dfrac{ G(x-y) \bF(y)}{\int_{\mathbb{R}} G(x-z) \bF(z)\, dz},\quad x,y\in \mathbb{R},
\end{equation}
and it can be interpreted as the transition probability from trait $y$ to trait $x$. The fact that it is a probability distribution function, $\int P(x;y)\, dy = 1$, is immediate by the choice of the normalization, which is such that constant functions $u_n \equiv const$ are invariant by the flow. 

Next, it can be proven that, if $\bF$ is strongly log-concave, then we have
\begin{equation}\label{E-contraction-linear-1}
W_\infty(P(x;\cdot), P(x',\cdot)) \leq \kappa\, |x-x'| ,
\end{equation}
where $\kappa\in(0,1)$ is related to the modulus of convexity of $\bV = -\log \bF$. By duality, this backward contraction estimate results in the forward estimate below, 
$$
\left\Vert \dfrac{d}{dx} \left(\log u_n\right)\right\Vert_{L^\infty} \leq \kappa \left\Vert \dfrac{d}{dx} \left(\log u_{n-1}\right)\right\Vert_{L^\infty},
$$
which by iteration and using the $L^\infty$ relative Fisher information, it can be expressed as follows
\begin{equation}\label{E-contraction-linear-2}
\mathcal{I}_\infty(F_n\Vert   \bF) \leq \kappa^n\, \mathcal{I}_\infty(F_{0}\Vert  \bF).
\end{equation}
As mentioned above, the key estimate \eqref{E-contraction-linear-1} is a consequence of the maximum principle on the Monge-Amp\`ere equation for the optimal transportation plan between $P(x;\cdot)$ and $P(x';\cdot)$. Interestingly, this is an argument borrowed from the theory of conservative equations, whereas our problem is not. The trick is to match an individual to its ancestor, which is obviously a conservative process, backward in time. 
\medskip

\paragraph*{\em Analogy with the Bakry-Emery argument.} There is some analogy between our results and the standard Bakry-Emery method for exponential relaxation towards equilibrium for the gradient flow of some displacement convex ``entropy'' \cite{B-94,AMTU-01,V-03,BGL-14}. Indeed, from \eqref{E-main-one-step-contraction} (alternatively \eqref{E-contraction-linear-2} in the linear case) we obtain exponential convergence on a quantity which is the $L^\infty$ analog of the usual ($L^2$) relative Fisher information,
\begin{equation}\label{E-main-Fisher-information-2}
\mathcal{I}_2(P\Vert Q):=\int_{\mathbb{R}} \left | \dfrac{d}{dx} \left(\log \frac{P}{Q}\right)(x)\right|^2 P(x) \, dx.
\end{equation}
Recall that, in the usual Bakry-Emery argument, the exponential convergence is established at the level of the dissipation of entropy, that is, the relative Fisher information \cite{V-03}. The exponential relaxation of the dissipation is intimately linked with the displacement convexity of the entropy functional (essentially because the gradient flow is differentiated, making appear the second derivative of the entropy functional). In our argument, it is the convexity of $\bV = -\log \bF$ which  induces the geometrical relaxation of the uniform relative Fisher information. 
\medskip

\paragraph*{\em Connection with another projective metric.} The uniform relative Fisher information \eqref{E-main-Fisher-information-infty} may also be viewed as a kind of first order version of the {\it Hilbert's projective distance} associated to the cone of non-negative functions, that is, 
$$\mathfrak{H}(P,Q):=\mathrm{osc}\left(\log \frac{P}{Q}\right)\equiv \sup_{x\in \mathbb{R}}\log\frac{P(x)}{Q(x)} -\inf_{x\in \mathbb{R}}\log\frac{P(x)}{Q(x)}.$$
The latter distance is well-suited for the analysis of 1-positively homogeneous, order-preserving, operators \cite{N-88}. An obvious reason is the projective character of that metric \cite{N-94}, which makes it insensitive to the exponential growth (or decay) $\mathcal O(\blambda^n)$.  This character is also shared by $\mathcal{I}_\infty$ (in contrast with $\mathcal{I}_2$). 
\medskip

\paragraph*{\em A linear argument, even in the non-linear case.} The previous discussion focussed on the linear operator \eqref{E-operator-A} for the sake of clarity. Interestingly, the non-linear case under study \eqref{E-operator-T} also involves a linear argument when formulated backward in time. Similarly, define the relative distribution $u_n=\frac{F_n}{\blambda^n \bF}$, where the pair $(\blambda,\bF)$ is the strongly log-concave solution to \eqref{E-nonlinear-eigenproblem} from part $(i)$ in the main Theorem \ref{T-main}. Then, $u_n$ satisfies the following forward-in-time non-linear problem:
\begin{equation}\label{E-nonlinear-normalized-intro}
u_n(x)=\frac{1}{\Vert u_{n-1}\bF\Vert_{L^1}}\iint_{\mathbb{R}^d}P(x;x_1,x_2)\,u_{n-1}(x_1)\,u_{n-1}(x_2)\,dx_1\,dx_2,\quad n\in \mathbb{N},\quad x\in \mathbb{R},
\end{equation}
where the function $P(x;x_1,x_2)$ is explicitly defined as
\begin{equation}\label{E-transition-probability-intro}
P(x;x_1,x_2)=\frac{G\left(x-\frac{x_1+x_2}{2}\right)\,\bF(x_1)\,\bF(x_2)\,dx_1\,dx_2}{\iint_{\mathbb{R}^2}G\left(x-\frac{x_1'+x_2'}{2}\right)\,\bF(x_1')\,\bF(x_2')\,dx_1'\,dx_2'},\quad x\in \mathbb{R},\,(x_1,x_2)\in \mathbb{R}^2.
\end{equation}
Since $P$ is normalized with respect to the variables $(x_1,x_2)$, then it can be regarded as a Markov kernel with source $x\in \mathbb{R}$ and target $(x_1,x_2)\in \mathbb{R}^2$ representing the probability of transitioning from the traits of the parents $(x_1,x_2)$ to the traits of  the offspring $x$.
In Lemma \ref{L-contraction-transition-probability}, we prove the very same contraction estimate as in \eqref{E-contraction-linear-1} for the family of Markov kernels $P$ indexed by its first variable $x$. The key difference is that this Markov kernel makes the transition between $u_n$ and $u_{n-1}\otimes u_{n-1}$ due to the joint distribution of parental traits (the non-linearity, in fact). This is rescued by an appropriate tensorization property of the relative Fisher information, which is expressed in Lemma \ref{L-log-estimate}. 

\medskip

\paragraph*{\em A close-to-optimal result despite a non-optimal argument.} The rate of contraction $\frac{2}{1+2\alphalpha}$ coincides with the optimal one in the quadratic case (see  Remark \ref{R-quadratic-selection}). However, there is some non-optimal step in the proof.  Indeed, our key contraction estimate \eqref{E-contraction-linear-1} is a consequence of the maximum principle on the Monge-Amp\`ere equation satisfied by the Brenier transportation map between the joint distributions of the parental traits $(X_1,X_2)$ and $(X_1',X_2')$. There is some subtlety here to be noticed, as the contraction is set for the $L^\infty$ Wasserstein distance (maximum of the optimal transportation displacement), whereas the Brenier transportation map used in our argument is  optimal for the $L^2$ Wasserstein distance. Nevertheless, in the quadratic case, the transportation map is simply a translation, so that it comes with the same cost, measured either in (weighted) $L^2$ or in $L^\infty$.

\medskip

\noindent {\bf Organization of the paper}.

\smallskip

In Section \ref{S-sketch-one-step-contraction} we provide a sketch of the proof of the one-step contraction property in Theorem \ref{T-main}(ii) under an additional technical condition. In Section \ref{S-contractivity} we derive the fundamental contraction property of the one-step transition probability of the problem under the $W_{\infty,1}$ Wasserstein distance (see definition below), thus removing the technical condition in the above sketch of the proof. In Section \ref{S-truncated-problem} we analyze a truncated version of the time-marching problem \eqref{E-nonlinear-eigenproblem} to bounded intervals, which will be necessary in next part. Section \ref{S-equilibria} focuses on proving the existence of strongly log-concave solutions of the nonlinear eigen-problem \eqref{E-nonlinear-eigenproblem} as claimed in Theorem \ref{T-main}(i). In Section \ref{S-convergence-equilibrium} we prove asymptotic exponential growth of \eqref{E-nonlinear-eigenproblem} for restricted initial data in Theorem \ref{T-main}(iii). Finally, Appendices \eqref{A-intermediate-dualities}, \eqref{A-lower-bound-convolution}, \eqref{A-also-l1-balls} contain some technical results to alleviate the reading of the paper.

\medskip

\noindent {\bf Notation}.

\smallskip

\noindent $\bullet$ {\bf (Vector norms)} Along the paper, $\mathbb{R}^d$ will be endowed with the various $\ell_q$ norms, namely, for any $z=(z_1,\ldots,z_d)\in \mathbb{R}^d$ and any $1\leq q\leq \infty$ we denote
\begin{equation}\label{E-lq-norm}
\Vert z\Vert_q:=\left\{
\begin{array}{ll}
\displaystyle\left(\sum_{i=1}^d |z_i|^q\right)^{1/q}, & \mbox{if }1\leq q<\infty,\\
\displaystyle\max_{1\leq i\leq d}|z_i|, & \mbox{if }q=\infty.
\end{array}
\right.
\end{equation}
The associated $\ell_2$ and $\ell_\infty$ open balls centered at $0$ with radius $R>0$ are respectively denoted by
\begin{equation}\label{E-BR-QR-balls}
B_R:=\{z\in \mathbb{R}^d:\,\Vert z\Vert_2<R\}\quad \mbox{and}\quad Q_R:=\{z\in \mathbb{R}^d:\,\Vert z\Vert_\infty<R\}.
\end{equation}

\smallskip

\noindent $\bullet$ {\bf (Measure spaces)} We denote by $\mathcal{M}(\mathbb{R}^d)$ the space of finite Radon measures, endowed with the total variation norm, and $\mathcal{M}^+(\mathbb{R}^d)$ represents the cone of non-negative finite Radon measures. Similarly, $\mathcal{P}(\mathbb{R}^d)$ is the subspace of probability measures, endowed with the narrow topology except otherwise specified. 

\smallskip

\noindent $\bullet$ {\bf (Wasserstein metrics)} For any $1\leq p\leq \infty$, we define the $L^p$ Wasserstein space 
\begin{align*}
\mathcal{P}_p(\mathbb{R}^d)&:=\left\{P\in \mathcal{P}(\mathbb{R}^d):\,\int_{\mathbb{R}^d}|z|^p\,P(dz)<\infty\right\}, \quad \mbox{if }1\leq p<\infty,\\
\mathcal{P}_\infty(\mathbb{R}^d)&:=\left\{P\in \mathcal{P}(\mathbb{R}^d):\,\supp P\mbox{ is compact}\right\}.
\end{align*}
Similarly, we consider the $L^p$ Wasserstein metric associated with the $\ell_q$ vector norm of $\mathbb{R}^d$. Specifically, for any $P,Q\in \mathcal{P}(\mathbb{R}^d)$ and any $1\leq p,q\leq \infty$ we denote
\begin{align}\label{E-Wpq-distance}
\begin{aligned}
W_{p,q}(P,Q)&:=\left(\inf_{\gamma\in \Gamma(P,Q)}\int_{\mathbb{R}^{2d}}\Vert z-\tilde z\Vert_q^p\,\gamma(dz,d\tilde z)\right)^{1/p}, \quad \mbox{if }1\leq p<\infty,\\
W_{\infty,q}(P,Q)&:=\inf_{\gamma\in \Gamma(P,Q)}~ \underset{z,\tilde z\in \mathbb{R}^d}{\gamma\mbox{-}\esssup}~\Vert z-\tilde z \Vert_q,
\end{aligned}
\end{align}
where $\Gamma(P,Q)$ is the family of transference plan $\gamma\in \mathcal{P}(\mathbb{R}^d\times \mathbb{R})^d$ with marginals $P$ and $Q$. Whilst the $L^p$ Wasserstein distances could be infinitely-valued over $\mathcal{P}(\mathbb{R}^d)$, note that they take finite values over $\mathcal{P}_p(\mathbb{R}^d)$ at least, although not exclusively. In particular, note that the $L^\infty$ Wasserstein distances take finite values over distributions $P$ and $Q$ that only differ on a space translation independently on their supports being compact or not. For this reason, along the paper we shall not restrict to compactly supported distributions, but anyway in all our computations the involved $L^\infty$ Wasserstein distances will take finite values as it will become clear later in the proofs.

\section{Proof of the one-step contraction property}\label{S-sketch-one-step-contraction}

For the reader convenience, we provide first the main ingredients behind the proof of the fundamental one-step contraction property in Theorem \ref{T-main}(ii). Here, we shall assume that Theorem \ref{T-main}(i) holds true, {\it i.e.}, there exists a $\alphalpha$-log-concave solution $(\blambda,\bF)$ to \eqref{E-nonlinear-eigenproblem} with $\alphalpha$ given by \eqref{E-main-alpha}. We remark that its use will be crucial in our following argument, but its proof is not apparent with regards to classical approaches based on the application of the Krein-Ruthman theorem. For this reason, a major part of this paper is devoted to rigorously address this question, which will be introduced in full detail in Section \ref{S-equilibria} of this paper.

\subsection{Sharp log-concavity parameter} 

First, we elaborate on the precise value of $\alphalpha$ given in \eqref{E-main-alpha}. Specifically, we prove that it amounts to the sharpest possible log-concavity parameter of a generic solution $(\lambda,F)$ to \eqref{E-nonlinear-eigenproblem}. To this end, it is worthwhile to note that the nonlinear operator $\mathcal{T}$ in \eqref{E-operator-T} can be restated as the composition of a multiplicative operator and a double convolution operator, namely,
\begin{equation}\label{E-operator-T-restatement}
\mathcal{T}[F]=\frac{e^{-m}}{\Vert F\Vert_{L^1}}(G*\bar F*\bar F),
\end{equation}
for every $F\in L^1_+(\mathbb{R})\setminus\{0\}$, where we define $\bar F(x):=2F(2x)$ for $x\in \mathbb{R}$. The starting point is to realize that strong log-concavity is stable under convolutions. More specifically, we have the following well known result (see \cite[Proposition 7.1]{SW-14} for futher details).

\begin{lemma}[Stability of log-concavity under convolutions]\label{L-stability-log-concavity-convolutions}
Assume that $F_1,F_2\in L^1_+(\mathbb{R})$ verify that $F_i$ are $\gamma_i$-log-concave for some $\gamma_1,\gamma_2>0$. Then $F_1*F_2$ is also $\gamma$-log-concave for $\gamma>0$ given by
$$\frac{1}{\gamma}=\frac{1}{\gamma_1}+\frac{1}{\gamma_2}.$$
\end{lemma}

Let us remark that the above result could be applied to any couple of Gaussian distributions $F_1$ and $F_2$ with respective variances $\sigma^2_1$ and $\sigma_2^2$ since they are in particular $\gamma_i$-log-concave with parameters $\gamma_i=\frac{1}{\sigma^2_i}$ for $i=1,2$. In doing so one finds that the above result is consistent with the classical fact that the convolution $F_1*F_2$ of two Gaussian distributions is again Gaussian with variance $\sigma^2=\sigma_1^2+\sigma_2^2$.

In addition, note that the mortality function $m$ has been chosen $\betabeta$-convex by the hypothesis \eqref{H-main-m-1} in Theorem \ref{T-main}, and then $e^{-m}$ is $\betabeta$-log-concave. Since strong log-concavity is also preserved under multiplication, and $\bar F$ is $4\gamma$-log-concave whenever $F$ is $\gamma$-log-concave, then we obtain that log-concavity must also be preserved under the full operator $\mathcal{T}$.

\begin{lemma}[Stability of log-concavity under $\mathcal{T}$]\label{L-stability-log-concavity}
Assume that $F\in L^1_+(\mathbb{R})\setminus\{0\}$ is $\gamma$-log-concave for some $\gamma>0$. Then, $\mathcal{T}[F]$ is also $\delta$-log-concave for $\delta>0$ given by
$$\delta=\betabeta+\frac{2\gamma}{1+2\gamma}.$$
\end{lemma}

Thereby, log-concavity is preserved by the dynamics in \eqref{E-nonlinear}, and we also obtain that the sharpest log-concavity coefficient of the eigenfunction $\bF$ must be the one given in \eqref{E-main-alpha}.

\begin{lemma}[Propagation of log-concavity]

~

\begin{enumerate}[label=(\roman*)]
\item Assume that $F_0\in L^1_+(\mathbb{R})\setminus\{0\}$ is $\alphalpha_0$-log-concave for some $\alphalpha_0>0$. Then, the solution $\{F_n\}_{n\in \mathbb{N}}$ to the evolution problem \eqref{E-nonlinear} verifies that $F_n$ is $\alphalpha_n$-log-concave for $\alphalpha_n>0$ verifying the recurrence
\begin{equation}\label{E-alphan}
\alphalpha_n=\betabeta+\frac{2\alphalpha_{n-1}}{1+2\alphalpha_{n-1}},\quad n\in \mathbb{N}.
\end{equation}
\item Assume that $(\lambda,F)$ is any solution to the nonlinear eigenproblem \eqref{E-nonlinear-eigenproblem} and that $F$ is strongly log-concave. Then, $F$ is $\alphalpha$-log-concave with $\alphalpha$ given by \eqref{E-main-alpha}, that is, 
$$\alphalpha=\betabeta+\frac{2\alphalpha}{1+2\alphalpha}.$$
\end{enumerate}
\end{lemma}

\begin{proof}
Since (i) is clear by Lemma \ref{L-stability-log-concavity}, we just prove (ii). Recall that for any solution $(\lambda,F)$ of \eqref{E-nonlinear-eigenproblem} with $\gamma$-log-concave $F$, we can build $F_n(x)=\lambda^n F(x)$, which solves the evolution problem \eqref{E-nonlinear}. Therefore, the above applied to $\{F_n\}_{n\in \mathbb{N}}$ shows that $F$ is $\alphalpha_n$ log-concave for any $n\in \mathbb{N}$ with $\{\alphalpha\}_{n\in \mathbb{N}}$ verifying the recurrence \eqref{E-alphan} above and $\alphalpha_0=\gamma$. Since $\alphalpha_n\rightarrow \alphalpha$, then $F$ is also $\alphalpha$-log-concave.
\end{proof}

\subsection{The renormalized problem}\label{ss-renormalized-problem}
We introduce a renormalized version of the evolution problem \eqref{E-nonlinear}. Specifically, for any solution $\{F_n\}_{n\in \mathbb{N}}$ to \eqref{E-nonlinear} we renormalize by the strongly log-concave quasi-equilibrium $\blambda^n \bF$ granted in Theorem \ref{T-main}(i). Namely, we set
\begin{equation}\label{E-nonlinear-normalization}
u_n(x):=\frac{F_n(x)}{\blambda^n \bF(x)},\quad n\in \mathbb{N},\,x\in \mathbb{R}.
\end{equation}
By inspection, we obtain that $\{u_n\}_{n\in \mathbb{N}}$ must solve the evolution problem
\begin{equation}\label{E-nonlinear-normalized}
u_n(x)=\frac{1}{\Vert u_{n-1}\bF \Vert_{L^1}}\iint_{\mathbb{R}^2}P(x;x_1,x_2)\,u_{n-1}(x_1) \,u_{n-1}(x_2)\,dx_1\,dx_2,
\end{equation}
for any $x\in \mathbb{R}$, where $P(x;x_1,x_2)$ is the {\it one-step transition probability} of transitioning from the parental traits $(x_1,x_2)$ to the descendant trait $x$. More, specifically, $P(x;\cdot)\in L^1_+(\mathbb{R}^2)\cap\mathcal{P}(\mathbb{R}^2)$ is a probability density on two variables $(x_1,x_2)$ depending on the parameter $x\in \mathbb{R}$ which takes the form (recall the notation $\bF = e^{-\bV}$),
\begin{align}\label{E-transition-probability}
\begin{aligned}
&P(x; x_1,x_2):=\frac{1}{Z(x)} e^{-W(x; x_1,x_2)},\quad x\in \mathbb{R},\quad (x_1,x_2)\in \mathbb{R}^2,\\
&W(x; x_1,x_2):=\frac{1}{2}\left\vert x-\frac{x_1+x_2}{2}\right\vert^2+\bV(x_1)+\bV(x_2),\\
&Z(x):=\iint_{\mathbb{R}^2}e^{-W(x; x_1,x_2)}\,dx_1\,dx_2.
\end{aligned}
\end{align}

Inspired by our method in \cite{CLP-21-arxiv}, we plan to study the relaxation to zero of $\left \Vert \frac{d}{dx}(\log u_n)\right \Vert_{L^\infty}$ as $n$ grows. Nevertheless, contrarily to the aforementioned paper, we do not need to accumulate a large enough amount of generations in order to observe some ergodic behavior, but we rather find a precise contraction of such a quantity after a single step. 

\subsection{A nonlinear Kantorovich-type duality}

Our new approach exploits a nice nonlinear version of a Kantorovich-type ``duality'' which relates the $L^\infty$ transport distance to the Lipschitz norm of the log of test functions. This nonlinear extension is reminiscent of the usual Kantorovich duality theorem, which relates $L^1$ transport distance to the Lipschitz norm of test functions, see \cite[Theorem 6.1.1]{AGS-08}. To the best of our knowledge, this relation appears to be new. Moreover, it does not represent an isolated example but there is a full family of related inequalities interpolating between the (classical) $L^1$ result and the seemingly new) $L^\infty$ result, and which further adapt to $L^p$ transport distances, see Appendix \ref{A-intermediate-dualities}.

\begin{lemma}[$L^\infty$-type Kantorovich duality]\label{L-log-estimate}
Consider the one-step transition from $u_0$ to  $u_1$ in \eqref{E-nonlinear-normalized}, where it is assumed that $u_0\in C^1(\mathbb{R})$ with $u_0>0$ and $\frac{d}{dx}(\log u_0)\in L^\infty(\mathbb{R})$. Then, we have
\begin{equation}\label{E-log-estimate}
\vert \log u_1(x)-\log u_1(\tilde x)\vert\leq \left \Vert\frac{d}{dx} (\log u_0)\right \Vert_{L^\infty}\,W_{\infty,1}(P(x;\cdot),P(\tilde x;\cdot)),
\end{equation}
for any $x,\tilde x\in \mathbb{R}$. Here, the metric $W_{\infty,1}$ represents the $L^\infty$ Wasserstein distance associated with the $\ell_1$ norm, {\it cf.} \eqref{E-Wpq-distance}.
\end{lemma}

\begin{proof}
Set $x,\tilde x\in \mathbb{R}$ and define $f:=P(x;\cdot)$, $g:=P(\tilde x;\cdot)$ for simplicity. Assume that $W_{\infty,1}(f,g)<\infty$ (otherwise the inequality is obvious). Indeed, this will always be the case as we prove later in Section \ref{S-contractivity}. Then, consider any $\gamma\in \Gamma(f,g)$ minimizing the $W_{\infty,1}$ transport distance \eqref{E-Wpq-distance} and note that
\begin{align*}
u_1(x)&=\frac{1}{\Vert u_0\bF \Vert_{L^1}}\iint_{\mathbb{R}^2} u_0(x_1)\,u_0(x_2)\,\gamma(dx_1,dx_2,d\tilde x_1,d\tilde x_2)\\
&=\frac{1}{\Vert u_0\bF\Vert_{L^1}}\iint_{\mathbb{R}^2} \exp\bigg(\log u_0(x_1)-\log u_0(\tilde x_1)+\log u_0(x_2)-\log u_0(\tilde x_2)\bigg)\\
&\qquad\qquad\qquad\qquad  \times u_0(\tilde x_1)\,u_0(\tilde x_2)\,\gamma(dx_1,dx_2,d\tilde x_1,d\tilde x_2)\\
&\leq \frac{1}{\Vert u_0 \bF\Vert_{L^1}}\iint_{\mathbb{R}^2} \exp\bigg(\left \Vert \frac{d}{dx}(\log u_0)\right \Vert_{L^\infty}\Vert (x_1,x_2)-(\tilde x_1,\tilde x_2)\Vert_1\bigg)\\
&\qquad\qquad\qquad\qquad  \times \,u_0(\tilde x_1)\,u_0(\tilde x_2)\,\gamma(dx_1,dx_2,d\tilde x_1,d\tilde x_2)\\
&\leq \exp\bigg(\left \Vert \frac{d}{dx}(\log u_0)\right \Vert_{L^\infty}\,W_{\infty,1}(P(x;\cdot),P(\tilde x,\cdot))\bigg)\,u_1(\tilde x),
\end{align*}
where in the next-to-last line we have used the mean value theorem and in the last one we have exploited the fact that $\gamma$ is minimizer. Then, taking logarithm on each side of the above inequality ends the proof.
\end{proof}

\begin{remark}[The choice of $\ell_1$ norm]
We note that Lemma \ref{L-log-estimate} is a particular instance of Proposition \ref{P-nonlinear-dualities} in Appendix \ref{A-intermediate-dualities} which can be recovered by setting $d_1=1$, $d_2=2$, $q=1$ and
$$u(x_1,x_2):=u_0(x_1)\,u_0(x_2),\quad (x_1,x_2)\in \mathbb{R}^2.$$
However, the special choice $q=1$ (that is $\ell_1$ norms) is apparently less clear at this stage since in fact choosing any other $1\leq q\leq \infty$ would be possible in Proposition \ref{P-nonlinear-dualities} and it would yield more generally
\begin{equation}\label{E-log-estimate-general}
\vert \log u_1(x)-\log u_1(\tilde x)\vert\leq 2^{1/{q'}}\left \Vert\frac{d}{dx} (\log u_0)\right\Vert_{L^\infty}\,W_{\infty,q}(P(x;\cdot),P(\tilde x;\cdot)),
\end{equation}
for every $x,\tilde x\in \mathbb{R}$. Here, the metric $W_{\infty,q}$ represents the $L^\infty$ Wasserstein distance associated with the $\ell_q$ norm, {\it cf.} \eqref{E-Wpq-distance}. By the natural relation between $\ell_1$ and $\ell_q$ vector norms, we infer that the above estimate \eqref{E-log-estimate} is sharper than \eqref{E-log-estimate-general}, namely
$$W_{\infty,1}(P(x;\cdot),P(\tilde x;\cdot))\leq 2^{1/q'}W_{\infty,q}(P(x;\cdot),P(\tilde x;\cdot)).$$
Therefore, it is clear that whenever $q>1$ the additional factor $2^{1/{q'}}$ makes the one-step contraction factor in next section non-optimal as compared to the explicit one-step contraction for quadratic selection $m(x)=\frac{\betabeta}{2}|x|^2$, as illustrated in Remark \ref{R-contraction-transition-probability-quadratic} and more detailed later in Remark \ref{R-contraction-l1-from-l2}.
\end{remark}

\subsection{Contraction of the one-step transition probability}
\label{sec:contraction 1-step formal}

The last step of our argument requires showing that the mapping $x\in \mathbb{R}\mapsto P(x;\cdot)\in L^1_+(\mathbb{R}^2)\cap \mathcal{P}(\mathbb{R}^2)$ is a contraction with the space $\mathcal{P}(\mathbb{R}^2)$ endowed with the $W_{\infty,1}$ Wasserstein distance in \eqref{E-Wpq-distance}. Specifically, in the following result we we quantify the exact Lipschitz constant, which will account for the precise contraction factor in Theorem \ref{T-main}(ii).

\begin{lemma}[$W_{\infty,1}$-contraction]\label{L-contraction-transition-probability}
For $P=P(x;x_1,x_2)$ given in \eqref{E-transition-probability} the following inequality holds true
$$W_{\infty,1}(P(x;\cdot),P(\tilde x;\cdot))\leq \frac{2}{1+2\alphalpha}|x-\tilde x|,$$
for every $x,\tilde x\in \mathbb{R}$.
\end{lemma}

Before entering into the details of the proof of the above result, let us note that putting Lemmas \ref{L-log-estimate} and \ref{L-contraction-transition-probability} together automatically implies the following one-step contraction estimate 
\begin{equation}
\left\Vert \dfrac{d}{dx} \left(\log u_1\right)\right\Vert_{L^\infty} \leq \dfrac{2}{1+2\alphalpha} \left\Vert \dfrac{d}{dx} \left(\log u_0\right)\right\Vert_{L^\infty},
\end{equation}
which can be iterated and propagated into \eqref{E-main-one-step-contraction} in Theorem \ref{T-main}(ii) (at generation $n$), thus concluding this section. Nevertheless, we remark that Lemma \ref{L-contraction-transition-probability} is far from straightforward as one typically cannot even ensure that the above $W_{\infty,1}$ distance must be finite because the probability densities $P(x;\cdot)$ and $P(\tilde x;\cdot)$ are supported on the full plane $\mathbb{R}^2$. 

\begin{remark}[Quadratic selection]\label{R-contraction-transition-probability-quadratic}
In the case of quadratic selection $m(x)=\frac{\betabeta}{2}|x|^2$ studied in \cite{CLP-21-arxiv}, we recall from Remark \ref{R-quadratic-selection} that  $\bF=G_{0,\sigma^2}$ with $\sigma^2=\alphalpha^{-1}$. Therefore, one easily obtains that $P(x;\cdot)$ is a bivariate normal distribution $P(x;\cdot)=G_{\mu_x,\Sigma}$ with mean and covariance matrix determined by
$$\mu_x:=\frac{1}{1+2\alphalpha}(x,x),\qquad \Sigma^{-1}:=\left(\begin{array}{cc}
\frac{1}{4}+\alphalpha & \frac{1}{4}\\
\frac{1}{4} & \frac{1}{4} +\alphalpha
\end{array}\right).$$
Since $\Sigma$ is independent of $x$, and thus common to all Gaussians $P(x;\cdot)$, the transport cost just amounts to moving the center $\mu_x$ of $P(x;\cdot)$ to the center $\mu_{\tilde x}$ of $P(\tilde x;\cdot)$. More precisely, the optimal transport map from $P(x;\cdot)$ to $P(\tilde x;\cdot)$ is simply the translation $T(x_1,x_2)=(x_1,x_2)+\mu_{\tilde x}-\mu_x$. Therefore, in this particular explicit case we recover Lemma \ref{L-contraction-transition-probability} (with identity indeed):
$$W_{\infty,1}(P(x;\cdot),P(\tilde x;\cdot))=\Vert \mu_x-\mu_{\tilde x}\Vert_1=\frac{2}{1+2\alphalpha}|x-\tilde x|.$$
\end{remark}

The goal of this section is to prove Lemma \ref{L-contraction-transition-probability}. To alleviate the notation, along this section we fix $x,\tilde x\in \mathbb{R}$ with $x\neq \tilde x$, we name $z:=(x_1,x_2)\in \mathbb{R}^2$, and we set the following notation.

\begin{definition}\label{D-f-g-from-P}
We define $f,g\in L^1_+(\mathbb{R}^2)\cap \mathcal{P}(\mathbb{R}^2)$ given by
$$
f(z):=P(x;z)=\frac{1}{Z}e^{-W(z)},\quad g(z):=P(\tilde x;z)=\frac{1}{\tilde Z}e^{-\tilde W(z)},
$$
where the potentials $W$ and $\tilde W$, and the normalizing constants $Z$ and $\tilde Z$ are set as follows
\begin{align*}
&W(z):= W(x;z)=\frac{1}{2}\left\vert x-\frac{x_1+x_2}{2}\right\vert^2+\bV(x_1)+\bV(x_2),\\
&\tilde W(z):= W(\tilde x;z)=\frac{1}{2}\left\vert \tilde x-\frac{x_1+x_2}{2}\right\vert^2+\bV(x_1)+\bV(x_2),\\
&Z:=Z(x)=\iint_{\mathbb{R}^2} e^{-W(z)}\,dz,\quad \tilde Z:=Z(\tilde x)=\iint_{\mathbb{R}^2} e^{-\tilde W(z)}\,dz.
\end{align*}
\end{definition}

For any transport map $T:\mathbb{R}^2\longrightarrow \mathbb{R}^2$ with $T_\# f=g$, note that a possible strategy in order to estimate the $W_{\infty,1}$ distance is to compute an $L^\infty$ bound for the $\ell_1$ associated displacement, namely,
\begin{equation}\label{E-W1infty-vs-displacement}
W_{\infty,1}(f,g)\leq \Vert\,\Vert T-I\Vert_1\Vert_{L^\infty}.
\end{equation}
Whilst the choice of $T$ is somehow arbitrary at this point, a comfortable one is usually the Brenier map $T:\mathbb{R}^2\longrightarrow \mathbb{R}^2$ from the density $f$ to the density $g$, which is characterized as the unique transport map verifying $T_\# f=g$ and solving the Monge problem \cite{B-91}
$$\iint_{\mathbb{R}^2} \Vert T(z)-z\Vert_2^2\,f(z)\,dz=W_{2,2}^2(f,g),$$
where $W_{2,2}$ is the $L^2$ Wasserstein distance associated with the $\ell_2$ norm of $\mathbb{R}^2$, {\it cf.} \eqref{E-Wpq-distance}. We remark that in many cases we do not lose any generality since the $W_{\infty,1}$ and the uniform bound of the $\ell_1$ displacement of the Brenier map have the same order, as depicted in the example of the translations treated in Remark \ref{R-quadratic-selection}, where the transport cost was indeed identical to the displacement.

Our proof of Lemma \ref{L-contraction-transition-probability} is based on the derivation of a novel $L^\infty$ bound of the $\ell_1$ displacement $\Vert T-I\Vert_1$ associated to the Brenier map $T$ between the densities $f$ and $g$. We derive those bounds by reformulating such a Brenier map as a solution to a Monge-Amp\`ere equation and using a version of {\it Caffarelli's maximum principle} along with the strong log-concavity of our densities. Indeed, by the strong log-concavity of $\bF$ in Theorem \ref{T-main}(i) we have
$$-D^2_{(x_1,x_2)}\log f=-D^2_{(x_1,x_2)}\log g\geq \left(\begin{array}{cc}
\frac{1}{4}+\alphalpha & \frac{1}{4}\\
\frac{1}{4} & \frac{1}{4} +\alphalpha
\end{array}\right)\geq \alphalpha \left(\begin{array}{cc}
1 & 0\\
0 & 1
\end{array}\right),$$
and then $f,g$ are $\alphalpha$-log-concave. The aforementioned strategy was previously applied in {\it Caffarelli's contraction principle} \cite{C-00,C-02,CF-21,CFJ-17} to find Lipschitz bounds of the Brenier map between strongly log-concave probability densities, and also in \cite{FS-21} to obtain $L^\infty$ bounds of the $\ell_2$ displacement associated with the Brenier map between two strongly log-concave densities supported on an Euclidean ball and bounded away from zero on it. We remark that all those results were derived under the Euclidean $\ell_2$ norm of $\mathbb{R}^2$, which is not convenient in our setting in view of the definition \eqref{E-Wpq-distance} of $W_{\infty,1}$. For the $\ell_1$ norm, we obtain new bounds on the Monge-Amp\`ere equation, which are able to find the sharp contraction factor, and which cannot be recovered by interpolation from known $\ell_2$ estimates, see Remark \ref{R-contraction-l1-from-l2}.

For the reader's convenience, we provide below a formal proof of Lemma \ref{L-contraction-transition-probability} under the strong additional assumption that the maximal $\ell_1$ displacement associated with the Brenier map is attained. Whilst true in particular situations ({\it cf.} Remark \ref{R-contraction-transition-probability-quadratic}), unfortunately this hypothesis is not necessarily always true, and thus the rigorous derivation requires further work which we provide in detail in Section \ref{S-contractivity}.

\begin{proof}[Formal proof of Lemma \ref{L-contraction-transition-probability}]
It is well known that the Brenier map $T:\mathbb{R}^2\longrightarrow\mathbb{R}^2$ from $f$ to $g$ take the form $T=\nabla\phi$ for some convex function $\phi:\mathbb{R}^2\longrightarrow\mathbb{R}$. Since $f,g>0$ and $f,g\in C^\infty(\mathbb{R}^2)$, then the regularity results in \cite{C-92-1} imply that $\phi\in C^\infty(\mathbb{R}^2)$. Moreover, the change of variable formula implies
\begin{equation}\label{E-Monge-Ampere}
\det(D^2\phi)=\frac{f}{g\circ \nabla\phi},\quad z\in \mathbb{R}^2.
\end{equation}
As usual we make the change of variables through the displacement potential
\begin{equation}\label{E-potential-psi}
\psi(z):=\phi(z)-\frac{1}{2}\Vert z\Vert_2^2,\quad z\in \mathbb{R}^2.
\end{equation}
In view of the relation \eqref{E-W1infty-vs-displacement}, we note that the core of the proof then reduces to obtaining $L^\infty$ bounds for the $\ell_1$ norm of the displacement of the Brenier map, that is,
\begin{equation}\label{E-displacement-H}
H(z):=\Vert T(z)-z\Vert_1=\Vert \nabla\psi(z)\Vert_1=|\partial_{x_1}\psi(z)|+|\partial_{x_2}\psi(z)|,\quad z\in \mathbb{R}^2.
\end{equation}
We start by restating the Monge-Amp\`ere equation \eqref{E-Monge-Ampere} by taking its logarithm,
\begin{equation}\label{E-Monge-Ampere-log}
\log\det (D^2\psi(z)+I)=\tilde W(\nabla\psi(z)+z)-W(z)+\log\frac{\tilde Z}{Z},\quad z\in \mathbb{R}^2.
\end{equation}
Taking partial derivatives $\partial_{x_k}$ in \eqref{E-Monge-Ampere-log} we have
\begin{equation}\label{E-Monge-Ampere-log-derivative}
{\rm tr}\left((D^2\phi)^{-1}\partial_{x_k} D^2\psi\right)=\nabla \tilde W(\nabla\psi+z)\cdot \partial_{x_k}\nabla\psi+(\nabla \tilde W(\nabla\psi+z)-\nabla W)\cdot e_k,\quad z\in\mathbb{R}^2,
\end{equation}
for $k=1,2$. Let us assume the simpler case that $H$ attains its maximum at some $z^*=(x_1^*,x_2^*)\in \mathbb{R}^2$ (for the general case we refer to Section \ref{S-contractivity}) and let us also define the auxiliary function
\begin{equation}\label{E-displacement-H-auxiliary}
\tilde H(z):=\sgn(\partial_{x_1}\psi(z^*))\,\partial_{x_1}\psi(z)+\sgn(\partial_{x_2}\psi(z^*))\,\partial_{x_2}\psi(z),\quad z\in \mathbb{R}^2.
\end{equation}
Then, $\tilde H$ must also attain its maximum at $z^*$ and it agrees with the maximum of $H$.  In particular, we have the necessary optimality conditions
\begin{equation}\label{E-optimality-conditions-H-tilde}
\nabla\tilde H(z^*)=0,\quad D^2 \tilde H(z^*)\leq 0.
\end{equation}
Now, we perform an appropriate convex combination of  \eqref{E-Monge-Ampere-log-derivative} depending on the signs of $\partial_{x_1}\psi(z^*)$ and $\partial_{x_2}\psi(z^*)$ in order to make the auxiliary function $\tilde H$ in \eqref{E-displacement-H} appear.

\medskip

$\diamond$ {\sc Case 1}:  $\partial_{x_1}\psi(z^*)\geq 0$ and $\partial_{x_2}\psi(z^*)\geq 0$.\\
In this case we have $\tilde{H}:=\partial_{x_1}\psi+\partial_{x_2}\psi.$ Evaluating \eqref{E-Monge-Ampere-log-derivative} at $z^*$ and summing over $k\in \{1,2\}$ we have
\begin{align*}
{\rm tr} ((D^2\phi(z^*))^{-1}D^2 \tilde{H}(z^*))&=\nabla \tilde W(\nabla\psi(z^*)+z^*)\cdot \nabla \tilde{H}(z^*)\\
&+(\nabla \tilde W(\nabla\psi(z^*)+z^*)-\nabla W(z^*))\cdot (1,1).
\end{align*}
By the optimality conditions \eqref{E-optimality-conditions-H-tilde} and since $D^2\phi(z^*)^{-1}$ is positive definite, the term in the left hand side above is non-positive (thus negligible) and we obtain
\begin{align*}
(\nabla \tilde W(\nabla\psi(z^*)+z^*)-\nabla \tilde W(z^*))\cdot (1,1)&\leq \nabla(W-\tilde W)(z^*)\cdot (1,1)=\tilde x-x.
\end{align*}
By expanding the left hand side we obtain
\begin{align*}
(\nabla &\tilde W(\nabla\psi(z^*)+z^*)-\nabla \tilde W(z^*))\cdot (1,1)\\
&=\frac{\partial_{x_1}\psi(z^*)+\partial_{x_2}\psi(z^*)}{2}+\bV'(\partial_{x_1}\psi(z^*)+x_1^*)-\bV'(x_1^*)+\bV'(\partial_{x_2}\psi(z^*)+x_2^*)-\bV'(x_2^*)\\
&\geq \frac{\partial_{x_1}\psi(z^*)+\partial_{x_2}\psi(z^*)}{2}+\alphalpha(\partial_{x_1}\psi(z^*)+\partial_{x_2}\psi(z^*))=\frac{1+2\alphalpha}{2}\tilde{H}(z^*),
\end{align*}
where we have used that in this case $\partial_{x_1}\psi(z^*)\geq 0$ and $\partial_{x_2}\psi(z^*)\geq 0$, along with the $\alphalpha$-convexity of $\bV$. Therefore, we conclude that $\tilde x> x$ and
$$\Vert H\Vert_{L^\infty}=H(z^*)=\tilde{H}(z^*)\leq \frac{2}{1+2\alphalpha}|x-\tilde x|.$$

\medskip

$\diamond$ {\sc Case 2}: $\partial_{x_1}\psi(z^*)< 0$ and $\partial_{x_2}\psi(z^*)<0$.\\
This case follows the same argument as {\sc Case 1}. Indeed, note now that $\tilde{H}=-\partial_{x_1}\psi-\partial_{x_2}\psi$. Then, we sum over $k\in {1,2}$, multiply by $-1$ on \eqref{E-Monge-Ampere-log-derivative} and we obtain 
$$\frac{1+2\alphalpha}{2}\tilde{H}(z^*)\leq x-\tilde x.$$
Hence, in this case we obtain $x>\tilde x$ and we recover
$$\Vert H\Vert_{L^\infty}=H(z^*)=\tilde{H}(z^*)\leq \frac{2}{1+2\alphalpha}|x-\tilde x|.$$

We show below that none of the other two possible cases cannot happen indeed.

\medskip

$\diamond$ {\sc Case 3}: $\partial_{x_1}\psi(z^*)\geq 0$ and $\partial_{x_2}\psi(z^*)<0$.\\
Our goal is to show that this case cannot take place. In this case, we have $\tilde{H}:=\partial_{x_1}\psi-\partial_{x_2}\psi$. Taking the difference of \eqref{E-Monge-Ampere-log-derivative} with $k=1$ and $k=2$ we obtain
\begin{align*}
{\rm tr} ((D^2\phi(z^*))^{-1}D^2 \tilde{H}(z^*))&=\nabla \tilde W(\nabla\psi(z^*)+z^*)\cdot \nabla \tilde{H}(z^*)\\
&+(\nabla \tilde W(\nabla\psi(z^*)+z^*)-\nabla W(z^*))\cdot (1,-1).
\end{align*}
Since $z^*$ is now a maximizer of $\tilde{H}$ we have
\begin{align*}
(\nabla \tilde W(\nabla\psi(z^*)+z^*)-\nabla \tilde W(z^*))\cdot (1,-1)&\leq \nabla(W-\tilde W)(z^*)\cdot (1,-1)=0
\end{align*}
The expansion on the left hand side is now radically different because the above factor $\frac{\partial_{x_1}\psi(z^*)+\partial_{x_2}\psi(z^*)}{2}$ cancels and now we obtain
\begin{align*}
(\nabla &\tilde W(\nabla\psi(z^*)+z^*)-\nabla \tilde W(z^*))\cdot (1,-1)\\
&=\bV'(\partial_{x_1}\psi(z^*)+x_1^*)-\bV'(x_1^*)-\bV'(\partial_{x_2}\psi(z^*)+x_2^*)+\bV'(x_2^*)\\
&\geq \alphalpha(\partial_{x_1}\psi(z^*)-\partial_{x_2}\psi(z^*))=\alphalpha\tilde{H}(z^*),
\end{align*}
which implies $\Vert H\Vert_{L^\infty}=H(z^*)=\tilde{H}(z^*)=0$. This is clearly impossible since otherwise $T(z)=z$ for all $z\in \mathbb{R}^2$, that is, $x=\tilde x$.

\medskip

$\diamond$ {\sc Case 4}: $\partial_{x_1}\psi(z^*)<0$ and $\partial_{x_2}\psi(z^*)\geq 0$.\\
This case cannot happen either thanks to the same argument as in {\sc Case 3} with $\tilde{H}$ replaced by $\tilde{H}=-\partial_{x_1}\psi+\partial_{x_2}\psi$. Then, we omit the proof.
\end{proof}

\subsection{Proof of the one-step contraction property}

With all the above machinery in hand, we are finally in position to prove the one-step contraction property \eqref{E-main-one-step-contraction} in Theorem \ref{T-main}.

\begin{proof}[Proof of Theorem \ref{T-main}(ii)]
Combining Lemmas \ref{L-log-estimate} and \ref{L-contraction-transition-probability} applied to the solution \eqref{E-nonlinear-normalization} of \eqref{E-nonlinear-normalized} we obtain
$$\left\Vert \frac{d}{dx} \left(\log \frac{F_n}{\bF}\right)\right\Vert_{L^\infty}\leq \frac{2}{1+2\alphalpha}\left\Vert \frac{d}{dx}\left(\log \frac{F_{n-1}}{\bF}\right)\right\Vert_{L^\infty},$$
for every $n\in \mathbb{N}$, and this amounts to \eqref{E-main-one-step-contraction}.
\end{proof}

\section{Main contractivity lemma}\label{S-contractivity}

In this section, we provide a rigorous proof of Lemma \ref{L-contraction-transition-probability}, where the {\it a priori} assumption that the maximal displacement associated with the Brenier map must be attained is no longer required. To do so, we shall argue by deriving a local version of the Lemma valid for strongly log-concave densities $f$ and $g$ compactly supported on an appropriate domain and bounded away from zero on it. More specifically, we propose to adapt the contribution of the maximum  principle to the formal argument above (Section \ref{sec:contraction 1-step formal}) to compact domains. However, since the maximum may be attained at the boundary, the boundary information is crucial in order to infer information from the non-linear elliptic PDE \eqref{E-Monge-Ampere} and therefore the choice of the domain cannot be made arbitrarily.

\subsection{Maximum principle under the $\ell_2$ norm}

Our starting point is the following result inspired by \cite{FS-21}, which underlines that the appropriate domain must be an Euclidean ball if one quantifies the displacement of the Brenier map in terms of $\ell_2$ norms of $\mathbb{R}^2$ or, more generally, $\mathbb{R}^d$.

\begin{proposition}[Maximum principle on $\ell_2$ balls]\label{P-maximum-principle-l2}
Consider two densities $f=e^{-W}$, $g=e^{-\tilde W}$ in $L^1_+(\mathbb{R}^d)\cap \mathcal{P}(\mathbb{R}^d)$ that are $\gamma$-log-concave for some $\gamma>0$. Assume that,
$$\{z\in \mathbb{R}^d:\,f(z)>0\}=\{z\in \mathbb{R}^d:\,g(z)>0\}=\bar B_R,$$
where $B_R$ is the Euclidean ball ({\it cf.} \eqref{E-BR-QR-balls}), and suppose that $f,g\in C^{1,\delta}(\bar B_R)$ for some $\delta>0$. Then, the Brenier map $T=\nabla\phi:\bar B_R\longrightarrow \bar B_R$ from $f$ to $g$ verifies
$$\Vert\,\Vert T-I\Vert_2\Vert_{L^\infty(\bar B_R)}\leq \frac{1}{\gamma}\Vert \,\Vert \nabla (W-\tilde W)\Vert_2\Vert_{L^\infty (\bar B_R)}.$$
\end{proposition}

\begin{proof}
Since $B_R$ is a uniformly convex domain, the classical regularity result by {\sc L. Caffarelli} for the Monge-Amp\`ere equation up to the boundary applies \cite{C-92-2, C-96}, and therefore $\phi\in C^{3,\delta}(\bar B_R)$. Again, we define the displacement potential $\psi(z)=\phi(z)-\frac{1}{2}\Vert z\Vert_2^2$ and the displacement function quantified in $\ell_2$ norms
$$H(z):=\frac{1}{2}\Vert T(z)-z \Vert_2^2=\frac{1}{2}\sum_{i=1}^d (\partial_{z_i}\psi(z))^2,\quad z\in \bar B_R.$$
Note that $H\in C^2(\bar B_R)$, and in particular it must attain its maximum at some point $z^*\in \bar B_R$. By \cite[Lemma 3.1]{FS-21} we have that $z^*\notin \partial B_R$, therefore it is an interior point and we obtain the necessary optimality conditions $\nabla H(z^*)=0$ and $D^2 H(z^*)\leq 0$. Since $\phi$ must solve the Monge-Amp\`ere equation
$$\det (D^2\phi)=\frac{f}{g\circ \nabla\phi},\quad z\in \bar B_R,$$
then we can argue similarly to the above formal proof of Lemma \ref{L-contraction-transition-probability} and obtain \eqref{E-Monge-Ampere-log-derivative}, that is,
$$
{\rm tr}\left((D^2\phi)^{-1}\partial_{z_k} D^2\psi\right)=\nabla \tilde W(\nabla\psi+z)\cdot \partial_{z_k}\nabla\psi+(\nabla \tilde W(\nabla\psi+z)-\nabla W)\cdot e_k,\quad z\in\bar B_R,
$$
for any $k=1,2,\ldots,d$. Multiplying the above by $\partial_{z_k}\psi$ and summing over $k\in \{1,2,\ldots,d\}$ we obtain
\begin{equation}\label{E-maximum-principle-l2-norm}
\sum_{k=1}^d{\rm tr}\left((D^2\phi)^{-1}\partial_{z_k} D^2\psi\right)\,\partial_{z_k}\psi =\sum_{k=1}^d\nabla \tilde W(\nabla\psi+z)\cdot \partial_{z_k}\nabla\psi\,\partial_{z_k}\psi+\left(\nabla \tilde W(\nabla\psi+z)-\nabla W\right)\cdot \nabla\psi.
\end{equation}
By inspection we obtain
\begin{align*}
\sum_k{\rm tr}\left((D^2\phi)^{-1}\partial_{z_k} D^2\psi\right)\,\partial_{z_k}\psi&={\rm tr}\left((D^2\phi)^{-1}D^2H\right)-\sum_k(\partial_{z_k}\nabla\psi)^\top\cdot (D^2\phi)^{-1}\cdot \partial_{z_k}\nabla\psi,\\
\sum_k\nabla \tilde W(\nabla\psi+z)\cdot \partial_{z_k}\nabla\psi\,\partial_{z_k}\psi&=\nabla \tilde W(\nabla\psi+z)\cdot \nabla H\\
\left(\nabla \tilde W(\nabla\psi+z)-\nabla W\right)\cdot \nabla\psi&=\left(\nabla \tilde W(\nabla\psi+z)-\nabla \tilde W\right)\cdot \nabla\psi+(\nabla \tilde W-\nabla W)\cdot \nabla\psi.
\end{align*}
Injecting the above expressions into the identity \eqref{E-maximum-principle-l2-norm}, evaluating at the maximizer $z^*$, and using the optimality conditions for $z^*$  along with the fact that $D^2\phi(z^*)^{-1}$ is positive definite, we obtain
\begin{equation}\label{E-maximum-principle-l2-norm-at-z*}
(\nabla \tilde W(\nabla\psi(z^*)+z^*)-\nabla \tilde W(z^*))\cdot \nabla\psi(z^*)\leq \nabla (W-\tilde W)(z^*)\cdot \nabla \psi(z^*).
\end{equation}
Note that $\tilde W$ is $\gamma$-convex, which by \eqref{E-maximum-principle-l2-norm-at-z*} and the Cauchy-Schwarz inequality infers
$$\gamma\Vert \nabla\psi(z^*)\Vert_2^2\leq \Vert\,\Vert \nabla (W-\tilde W)\Vert_2\Vert_{L^\infty(\bar B_R)}\Vert\nabla\psi(z^*)\Vert_2,$$
and this ends the proof.
\end{proof}

Whilst not directly applicable to $f,g$ given as in Definition \ref{D-f-g-from-P}, we may apply it to their truncations to any Euclidean ball.

\begin{definition}[Truncation to $\bar B_R$]\label{D-f-g-from-P-truncation-l2}
For the probability densities $f,g\in L^1_+(\mathbb{R}^2)\cap \mathcal{P}(\mathbb{R}^2)$ given in Definition \ref{D-f-g-from-P}, we define their truncations to the Euclidean ball $B_R$ ({\it cf.} \eqref{E-BR-QR-balls}) as follows
\begin{align*}
\begin{aligned}
f_R(z)&:=\frac{1}{Z_{R}}e^{-W_R(z)}, & g_R(z)&:=\frac{1}{\tilde Z_{R}}e^{-\tilde W_R(z)},\\
W_R(z)&:=W(z)+\chi_{\bar B_R}(z), & \tilde W_R(z)&:=\tilde W(z)+\chi_{\bar B_R}(z),\\
Z_R&:=\int_{\mathbb{R}^2}e^{-W_R(z)}\,dz, & \tilde Z_R&:=\int_{\mathbb{R}^2}e^{-\tilde W_R(z)}\,dz,
\end{aligned}
\end{align*}
for any $R>0$. Here, $\chi_A:\mathbb{R}^d\longrightarrow (-\infty,+\infty]$ denotes the characteristic function of convex analysis associated with any set $A\subset \mathbb{R}^d$, that is, 
$$\chi_A(z):=\left\{\begin{array}{ll}
0, & \mbox{if}\quad z\in A,\\
+\infty, & \mbox{if}\quad z\in \mathbb{R}^d\setminus A.
\end{array}\right.$$ 
\end{definition} 

\begin{remark}[Application to our case study]\label{R-contraction-l1-from-l2}
Note that the truncated densities $f_R,g_R\in C^\infty(\bar B_R)$ in Definition \ref{D-f-g-from-P-truncation-l2} are compactly supported on $\bar B_R$ and bounded away from zero on it, and in addition,
$$D^2W(x_1,x_2)=D^2 \tilde W(x_1,x_2)=\left(\begin{array}{cc}
\frac{1}{4}+\bV''(x_1) & \frac{1}{4}\\
\frac{1}{4} & \frac{1}{4}+\bV''(x_2)
\end{array}\right)\geq \left(\begin{array}{cc}
\alphalpha	& 0\\
0 & \alphalpha
\end{array}\right).$$
Hence, $f_R$ and $g_R$ are $\alphalpha$-log-concave, and Proposition \ref{P-maximum-principle-l2} can be applied, thus yielding a uniform bound over $\bar B_R$ of the displacement associated with the Brenier map $T_R:\bar B_R\longrightarrow \bar B_R$ sending $f_R$ to $g_R$. Namely,
$$\Vert\Vert T_R-I\Vert_2\Vert_{L^\infty(\bar B_R)}\leq \frac{1}{\alphalpha}\Vert\,\Vert\nabla (W-\tilde W)\Vert_2\Vert_{L^\infty(\bar B_R)}=\frac{1}{\alphalpha}\left\Vert \frac{1}{2}(\tilde x-x,\tilde x-x)\right\Vert_2=\frac{1}{\sqrt{2}\alphalpha}|x-\tilde x|,$$
and by interpolation we have
$$W_{\infty,1}(f_R,g_R)\leq \sqrt{2}\,W_{\infty,2}(f_R,g_R)\leq \sqrt{2}\,\Vert\Vert T_R-I\Vert_2\Vert_{L^\infty(\bar B_R)}\leq \frac{1}{\alphalpha}|x-\tilde x|.$$
In particular, we note that such an estimate only provides contraction as long as $\alphalpha>1$ and, in addition, the contraction factor is worse that the one claimed in Lemma \ref{L-contraction-transition-probability} as depicted in Figure \ref{fig:rates-contraction-l1-vs-l2}.
\end{remark}

\begin{figure}
\begin{tikzpicture}[scale=0.9]
\begin{axis}[
  axis x line=middle, axis y line=middle,
  label/.style={at={(current axis.right of origin)},anchor=west},
  xmin=0, xmax=5, xtick={0,0.5,1,2,3,4,5}, xlabel=$\alphalpha$,
  ymin=0, ymax=2.25, ytick={0,0.25,0.5,0.75,1,1.25,1.5,1.75,2},
]
\addplot [
    domain=0.5:5, 
    samples=100, 
    color=blue,
    line width=0.4mm,
]{2/(1+2*x)};
\addlegendentry{$\frac{2}{1+2\alphalpha}$}
\addplot [
    domain=0.5:5, 
    samples=100, 
    color=red,
    line width=0.4mm,
]{1/x};
\addlegendentry{$\frac{1}{\alphalpha}$}
\addplot [
    domain=0.5:5, 
    samples=100, 
    color=red,
    line width=0.4mm,
]{1/x};
\addplot [
    domain=0:5, 
    samples=100, 
    color=black,
    dashed,
    line width=0.3mm,
]{1};
\addplot[
	domain=0:5,
	samples=50,
	color=black,
	dashed,
	line width=0.3mm	
] coordinates {(0.5,0)(0.5,2.5)};
\addplot[
	domain=0:5,
	samples=50,
	color=black,
	dashed,
	line width=0.3mm	
] coordinates {(0,2)(0.5,2)};
\addplot [mark=*, mark size=1.5pt] coordinates {(0.5,1)};
\addplot [mark=*, mark size=1.5pt] coordinates {(0.5,2)};
\addplot [mark=*, mark size=1.5pt] coordinates {(1,1)};
\end{axis}
\end{tikzpicture}
\caption{Comparison of the theoretical contraction factor $\frac{1}{1+2\alphalpha}$ in Lemma \ref{L-contraction-transition-probability}, and the contraction factor $\frac{1}{\alphalpha}$ obtained in Remark \ref{R-contraction-l1-from-l2} as an application of Proposition \ref{P-maximum-principle-l2} and an interpolation of $\ell_1$ norms by $\ell_2$ norms of $\mathbb{R}^2$.}
\label{fig:rates-contraction-l1-vs-l2}
\end{figure}
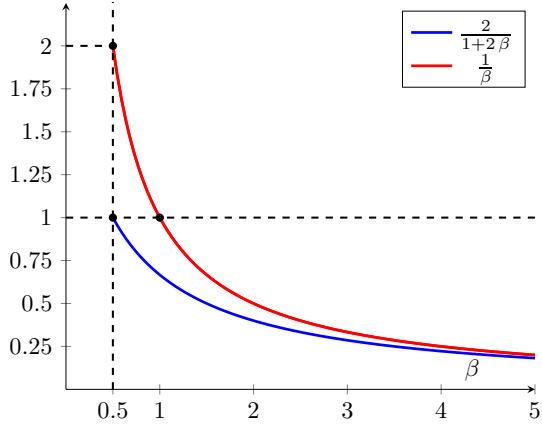

In the following section, we propose to correct such a defect, which appears to be due to the misuse of $\ell_2$ norms in the maximum principle of Proposition \ref{P-maximum-principle-l2}. More specifically, we propose a novel version of the maximum principle which is better adapted to $\ell_1$ norms of the displacement.

\subsection{Maximum principle under the $\ell_1$ norm}

A fundamental step of the proof of Proposition \ref{P-maximum-principle-l2} was the observation made in \cite[Lemma 3.1]{FS-21}. Specifically, for $f$ and $g$ compactly supported on a $\ell_2$ ball and bounded away from zero on it, the maximal $\ell_2$ displacement of the Brenier map must be attained at some interior point in the ball. Apparently, the use of $\ell_2$ norms to quantify the size of the displacement proved extremely well suited in order to control the boundary information on $\ell_2$ balls. Interestingly, in the sequel we show that in order to find precise information about the maximizers for the $\ell_1$ displacement, we need our densities $f$ and $g$ to be supported over $\ell_\infty$ balls $\bar B_R$ ({\it cf.} \eqref{E-BR-QR-balls}). This is the content of the following

\begin{lemma}[Maximizers in the $\ell_1$ setting]\label{L-properties-maximizers-l1}
Consider two densities $f,g\in L^1_+(\mathbb{R}^2)\cap \mathcal{P}(\mathbb{R}^2)$, assume that,
$$\{z\in \mathbb{R}^2:\,f(z)>0\}=\{z\in \mathbb{R}^2:\,g(z)>0\}=\bar Q_R,$$
where $Q_R$ is the $\ell_\infty$ ball ({\it cf.} \eqref{E-BR-QR-balls}), and suppose that $f,g\in C^{1,\delta}(\bar B_R)$ for some $\delta>0$. Let $T=\nabla\phi:\bar B_R\longrightarrow \bar B_R$ be the Brenier map from $f$ to $g$, define the displacement potential $\psi(z):=\phi(z)-\frac{1}{2}\Vert z\Vert_2^2$ and the displacement function quantified in $\ell_1$ norm
\begin{equation}\label{E-displacement-H-local}
H(z):=\Vert T(z)-z\Vert_1=|\partial_{x_1}\psi(z)|+|\partial_{x_2}\psi(z)|,\quad z\in \bar Q_R,
\end{equation}
Then, $T\in C^{2,\delta}(\bar Q_R)$ and we have  the optimality conditions
\begin{equation}\label{E-optimality-conditions-H-tilde-local}
\nabla \tilde H(z^*)=0,\quad D^2 \tilde H(z^*)\leq 0,
\end{equation}
for any maximizer $z^*=(z_1^*,z_2^*)\in \bar Q_R$ of $H$, where $\tilde{H}$ is the auxiliary function
\begin{equation}\label{E-displacement-H-auxiliary-local}
\tilde H(z):=\sgn(\partial_{x_1}\psi(z^*))\,\partial_{x_1}\psi(z)+\sgn(\partial_{x_2}\psi(z^*))\,\partial_{x_2}\psi(z),\quad z\in \bar Q_R.
\end{equation}
\end{lemma}

In contrast with Proposition \ref{P-maximum-principle-l2}, $Q_R$ is not uniformly convex. Then, the regularity theory of the Monge-Amp\`ere equation is not directly applicable in full generality. Specifically, since $f,g\in C^{1,\delta}(\bar Q_R)$ are bounded away from zero on $\bar Q_R$, then $T\in C^{0,\delta}(\bar Q_R)$ by \cite{C-92-2}. However, the lack of uniform convexity may prevent the full elliptic regularity \cite{C-96}, which claims that $T$ is a diffeomorphism of class $C^{2,\delta}(\bar Q_R)$. Fortunately, we can proceed as in \cite[Theorem 3.3]{J-19} which, thanks to a clever symmetrization argument around each corner of $Q_R$ and the classical interior regularity in \cite{C-92-1}, shows that $T$ is indeed a diffeomorphism of class $C^{2,\delta}(\bar Q_R)$. Moreover, it fixes the corners and sends each segment of the boundary to itself. This guarantees in particular that $\tilde H\in C^2(\bar Q_R)$ and the optimality conditions above make sense, as shown below. 

\begin{proof}[Proof of Lemma \ref{L-properties-maximizers-l1}]
We remark that $z^*\in \bar Q_R$ must also be a maximizer of $\tilde H$ since we have
$$\tilde H(z)\leq H(z)\leq H(z^*)=\tilde H(z^*),$$
for every $z^*\in \bar Q_R$ by the definition of $H$ and $\tilde H$ in \eqref{E-displacement-H-local} and \eqref{E-displacement-H-auxiliary-local}. Since the maximizer $z^*$ may lie in principle in all $\bar Q_R$, two possible options arise, either $z^*\in Q_R$ or $z^*\in \partial Q_R$. In the first case, the usual optimality conditions at interior points yield \eqref{E-optimality-conditions-H-tilde-local}. In the second case, namely $z^*\in \partial Q_R$, note that the result is trivial if $z^*$ is one of the four corners since those are fixed points of $T$ and therefore $\tilde{H}\equiv 0$. Hence, here on we will assume that $z^*\in \partial Q_R$ is not at a corner, but it lies in the interior of some of the four segments. Note that at those points we only have to prove that $\nabla\tilde{H}(z^*)=0$. In fact, we remark that those $z^*$ can be approached by interior points from any direction, and then the above readily implies the second order optimality condition $D^2 \tilde{H}(z^*)\leq 0$. To show that $\nabla \tilde H(z^*)=0$, note that the boundary $\partial Q_R$ contains four segments:
\begin{align*}
S_1^+&:=\{(x_1,x_2)\in \mathbb{R}^2:\,x_1=R,\,x_2\in [-R,R]\},\\
S_1^-&:=\{(x_1,x_2)\in \mathbb{R}^2:\,x_1=-R,\,x_2\in [-R,R]\},\\
S_2^+&:=\{(x_1,x_2)\in \mathbb{R}^2:\,x_1\in [-R,R],\,x_2=R\},\\
S_2^-&:=\{(x_1,x_2)\in \mathbb{R}^2:\,x_1\in [-R,R],\,x_2=-R\}.
\end{align*}
Since $T(\partial Q_R)=\partial Q_R$ and each segment is mapped to itself, then we have the following information
\begin{align}
\partial_{x_1}\psi(z)&=0,\qquad \mbox{if }z\in S_1^+\cup S_1^-,\label{E-transport-segments-1}\\
\partial_{x_2}\psi(z)&=0,\qquad \mbox{if }z\in S_2^+\cup S_2^-.\label{E-transport-segments-2}
\end{align}
By differentiation it is clear that we also have
\begin{equation}\label{E-transport-boundary-crossed-derivative}
\partial_{x_1 x_2}\psi(z)=0,\qquad \mbox{if }z\in \partial Q_R.
\end{equation}
Now, we argue according to the four possible segments of $\partial Q_R$ that $z^*$ may belong to. 

\medskip

$\diamond$ {\sc Case 1}: $z^*\in S_1^+\cup S_1^-$.\\
In this case, by \eqref{E-transport-segments-1} we have $\partial_{x_1}\psi(z^*)=0$ and therefore we have 
$$\tilde{H}(z)=\sgn(\partial_{x_2}\psi(z^*))\,\partial_{x_2}\psi(z),\quad z\in \bar Q_R.$$
Since $z^*$ is a maximizer of $\tilde H$, then there exist $\lambda\in \mathbb{R}$ (indeed $\lambda\geq 0$ if $z^*\in S_1^+$ and $\lambda\leq 0$ if $z^*\in S_1^-$) such that its gradient at $z^*$ equals the multiple $\lambda (1,0)$ of the outer normal vector, that is, 
$$\nabla\tilde H(z^*)=\sgn(\partial_{x_2}\psi(z^*))\left(\begin{array}{c} \partial_{x_1x_2}\psi(z^*)\\ \partial_{x_2x_2}\psi(z^*)\end{array}\right)=\left(\begin{array}{c} \lambda \\ 0\end{array}\right).$$
This implies that the second component of the gradient must vanish, but the first one also vanishes by the condition \eqref{E-transport-boundary-crossed-derivative} on the crossed derivative. Then, we have $\nabla \tilde{H}(z^*)=0$.

\medskip

$\diamond$ {\sc Case 2}: $z^*\in S_2^+\cup S_2^-$.\\
In this case, by \eqref{E-transport-segments-2} we have $\partial_{x_2}\psi(z^*)=0$ and therefore we have 
$$\tilde{H}(z)=\sgn(\partial_{x_1}\psi(z^*))\,\partial_{x_1}\psi(z),\quad z\in \bar Q_R.$$
Since $z^*$ is a maximizer of $\tilde H$, then there exist $\lambda\in \mathbb{R}$ (indeed $\lambda\geq 0$ if $z^*\in S_2^+$ and $\lambda\leq 0$ if $z^*\in S_2^-$) such that its gradient at $z^*$ equals the multiple $\lambda (0,1)$ of the outer normal vector, that is, 
$$\nabla\tilde H(z^*)=\sgn(\partial_{x_1}\psi(z^*))\left(\begin{array}{c} \partial_{x_1x_1}\psi(z^*)\\ \partial_{x_1x_2}\psi(z^*)\end{array}\right)=\left(\begin{array}{c} 0 \\ \lambda\end{array}\right).$$
This implies that the first component of the gradient must vanish, but the second one also vanishes by the condition \eqref{E-transport-boundary-crossed-derivative} on the crossed derivative. Then, we have $\nabla \tilde{H}(z^*)=0$.
\end{proof}

We remark that the unique formal point of the sketch of the proof of Lemma \ref{L-contraction-transition-probability} in Section \ref{S-sketch-one-step-contraction} which could break down is the fact that for the global densities $f$ and $g$ in Definition \ref{D-f-g-from-P} the $\ell_1$ displacement of their Brenier map does not attain its maximum necessarily. In particular, we may be deprived from the optimality condition \eqref{E-optimality-conditions-H-tilde}, which was crucially used throughout the maximum-type principle sketched in Section \ref{S-sketch-one-step-contraction}. However, Lemma \ref{L-properties-maximizers-l1} does guarantee that the maximum must be attained and the optimality conditions \eqref{E-optimality-conditions-H-tilde-local} must hold if $f$ and $g$ are replaced by similar truncated densities over $\ell_\infty$ balls. In fact, the result does not exploit the special potential $\bV$ of the eigenfunction $\bF$ in Definition \ref{D-f-g-from-P} for $f,g$, but it can actually be replaced by any strongly convex function supported on $\bar Q_R$.

\begin{lemma}[Maximum principle on $\ell_\infty$ balls]\label{L-maximum-principle-l1}
For any $\gamma$-convex potential $V\in C^{1,\delta}_{\rm loc}(\mathbb{R})$ with $\gamma>0$, any $x,\tilde x\in \mathbb{R}$ with $x\neq \tilde x$, and any $R>0$ we define $f,g\in L^1_+(\mathbb{R}^d)\cap \mathcal{P}(\mathbb{R}^2)$ given by
$$f(z)=\frac{1}{Z}e^{-W(z)},\quad g(z)=\frac{1}{\tilde Z}e^{-\tilde W(z)},\quad z\in \mathbb{R}^2,$$
where the potentials $W$ and $\tilde W$, and the normalizing constants $Z$ and $\tilde Z$ are set as follows
\begin{align*}
&W(z):=\frac{1}{2}\left\vert x-\frac{x_1+x_2}{2}\right\vert^2+V(x_1)+V(x_2)+\chi_{\bar Q_R}(z),\\
&\tilde W(z):=\frac{1}{2}\left\vert \tilde x-\frac{x_1+x_2}{2}\right\vert^2+V(x_1)+V(x_2)+\chi_{\bar Q_R}(z),\\
&Z:=\iint_{\mathbb{R}^2}e^{-W(z)}\,dz,\quad \tilde Z:=\iint_{\mathbb{R}^2} e^{-\tilde W(z)}\,dz.
\end{align*}
Then, the Brenier map $T=\nabla\phi:\bar Q_R\longrightarrow \bar Q_R$ from $f$ to $g$ verifies
$$\Vert\,\Vert T-I\Vert_1\Vert_{L^\infty(\bar Q_R)}\leq \frac{2}{1+2\gamma}|x-\tilde x|.$$
\end{lemma} 

As explained above, we omit the proof since it follows the formal proof of Lemma \ref{L-contraction-transition-probability} in Section \ref{S-sketch-one-step-contraction} and the optimality conditions in Lemma \ref{L-properties-maximizers-l1}. In particular, by setting $V=\bV$ and $\gamma=\alphalpha$ we have that Lemma \ref{L-maximum-principle-l1} is directly applicable to the truncations to $\bar Q_R$ of the densities $f,g$ in Definition \ref{D-f-g-from-P}.

\begin{definition}[Truncation to $\bar Q_R$]\label{D-f-g-from-P-truncation-l1}
For the probability densities $f,g\in L^1_+(\mathbb{R}^2)\cap \mathcal{P}(\mathbb{R}^2)$ given in Definition \ref{D-f-g-from-P}, we define their truncations to the $\ell_\infty$ ball $\bar Q_R$ ({\it cf.} \eqref{E-BR-QR-balls}) as follows
\begin{align*}
\begin{aligned}
f_R(z)&:=\frac{1}{Z_{R}}e^{-W_R(z)}, & g_R(z)&:=\frac{1}{\tilde Z_{R}}e^{-\tilde W_R(z)},\\
W_R(z)&:=W(z)+\chi_{\bar Q_R}(z), & \tilde W_R(z)&:=\tilde W(z)+\chi_{\bar Q_R}(z),\\
Z_R&:=\int_{\mathbb{R}^2}e^{-W_R(z)}\,dz, & \tilde Z_R&:=\int_{\mathbb{R}^2}e^{-\tilde W_R(z)}\,dz,
\end{aligned}
\end{align*}
for any $R>0$.
\end{definition}

Then, we are in position to rigorously prove Lemma \ref{L-contraction-transition-probability} by taking limits $R\rightarrow \infty$ and noting that Lemma \ref{L-maximum-principle-l1} yields a uniform bound of the displacement independent on $R$.

\begin{proof}[Rigorous proof of Lemma \ref{L-contraction-transition-probability}]
Consider $f$ and $g$ given in Definition \ref{D-f-g-from-P} and set the associated Brenier map $T:\mathbb{R}^2\longrightarrow\mathbb{R}^2$ from $f$ to $g$. Similarly, we consider the family of truncations $f_R$ and $g_R$ in Definition \ref{D-f-g-from-P-truncation-l1} and we set the associated Brenier maps $T_R:\mathbb{R}^2\longrightarrow \mathbb{R}^2$. By the above Lemma \ref{L-maximum-principle-l1} we have
\begin{equation}\label{E-maximum-principle-l1-truncation}
\Vert\,\Vert T_R-I\Vert_1\Vert_{L^\infty(\bar Q_R)}\leq \frac{2}{1+2\alphalpha}|x-\tilde x|,
\end{equation}
for every $R>0$. We set the optimal transference plans $\gamma\in \Gamma_o(f,g)$ and $\gamma_R\in \Gamma_o(f_R,g_R)$ associated with the $W_{2,2}$ distance, which are known to be supported on the graph of the above Brenier maps, {\it i.e.},
$$
\gamma:=(I,T)_{\#}f,\quad \gamma_R:=(I,T_R)_{\#}f_R.
$$
Since $W$ and $\tilde W$ are $\alphalpha$-convex, we have the enough integrability on $f$ and $g$ to ensure that $f,g\in \mathcal{P}_2(\mathbb{R}^2)$. Hence, the dominated convergence theorem applies and we have indeed
$$f_R\rightarrow f,\quad g_R\rightarrow f\quad \mbox{in}\quad (\mathcal{P}_2(\mathbb{R}^2),W_{2,2}).$$
By stability of optimal transference plans, the sequence $\gamma_R$ must converge narrowly to some optimal transference plan (up to a subsequence), see \cite[Proposition 7.1.3]{AGS-08}. Since the unique optimal transference plan between $f$ and $g$ is precisely the above $\gamma$ supported on the graph of $T$, then we obtain
$$\gamma_R\rightarrow \gamma\quad \mbox{narrowly in }\mathcal{P}(\mathbb{R}^2).$$
Now we use the Kuratowski convergence of the supports under the narrow convergence of measures, see \cite[Proposition 5.1.8]{AGS-08}. Namely, consider any $z\in \mathbb{R}^2$. Since $(z,T(z))\in \supp \gamma$, then there exists $(z^R,w^R)\in \supp \gamma_R$ such that $(z^R,w^R)\rightarrow (z,T(z))$. Since $\gamma_R$ is supported on the graph of $T_R$ then $z^R\in \bar Q_R$ and $w^R=T_R(z^R)$. In particular, we have $T_R(z^R)-z^R\rightarrow T(z)-z$ as $R\rightarrow \infty$ and by the above uniform bound \eqref{E-maximum-principle-l1-truncation} the same bound is preserved in the limit, that is,
$$W_{\infty,1}(f,g)\leq \Vert \,\Vert T-I\Vert_1\Vert_{L^\infty}\leq \frac{2}{1+2\alphalpha}|x-\tilde x|.$$
\end{proof}

\begin{remark}[Replacing $\ell_\infty$ balls by $\ell_1$ balls]
We note that in Lemmas \ref{L-properties-maximizers-l1} and \ref{L-maximum-principle-l1} the choice of $\ell_\infty$ balls may seem to be crucial. However, that choice is not essential and we refer to Appendix \ref{A-also-l1-balls} for an alternative version of those Lemmas \ref{L-maximum-principle-l1} for densities compactly supported over $\ell_1$ balls and bounded away from zero on them.
\end{remark}

\section{Analysis of a truncated problem}\label{S-truncated-problem}

In this part, we study an auxiliary version of the original time marching problem \eqref{E-nonlinear} restricted to the bounded interval $I_R:=(-R,R)$ with $R>0$, namely,
\begin{equation}\label{E-nonlinear-truncated}
F_n^R=\mathcal{T}_R[F_{n-1}^R],\quad n\in \mathbb{N},\,x\in \mathbb{R}.
\end{equation}
Here, we truncate the selection function $m_R$ as follows
\begin{equation}\label{E-selection-function-truncated}
m_R(x):=m(x)+\chi_{\bar I_R}(x),\quad x\in \mathbb{R},
\end{equation}
so that the truncated integral operator $\mathcal{T}_R$ takes the form
\begin{equation}\label{E-operator-T-truncated}
\mathcal{T}_R[F](x):=e^{-m_R(x)}\iint_{\mathbb{R}^2} G\left(x-\frac{x_1+x_2}{2}\right)\,F(x_1)\,\frac{F(x_2)}{\Vert F\Vert_{L^1}}\,dx_1\,dx_2,\quad x\in \mathbb{R}.
\end{equation}
Again, solutions of the form $F_n^R(x)=(\lambda^R)^n\,F^R(x)$ come as eigenpairs of the non-linear eigenproblem
\begin{align}\label{E-nonlinear-eigenproblem-truncated}
\begin{aligned}
&\lambda^R F^R=\mathcal{T}_R[F^R],\quad x\in \mathbb{R},\\
& F^R\geq 0,\quad \int_{\mathbb{R}} F^R(x)\,dx=1.
\end{aligned}
\end{align}
The goal of this section is to derive an analogous truncated version of Theorem \ref{T-main}. More specifically, we study: 1) Existence of a unique strongly log-concave solution $(\blambda^R,\bF^R)$ to \eqref{E-nonlinear-eigenproblem-truncated}, and 2) Quantitative relaxation of the solutions to \eqref{E-nonlinear-truncated} towards the quasi-equilibrium $(\blambda^R)^n\bF^R$.

\begin{theorem}[Truncated problem]\label{T-main-truncated}
Consider any $m\in C^2(\mathbb{R})$ verifying \eqref{H-main-m-1}-\eqref{H-main-m-2} in Theorem \ref{T-main}. Set any $R>0$ and define the truncation $m_R$ according to \eqref{E-selection-function-truncated}. Then, the following statements hold true:
\begin{enumerate}[label=(\roman*)]
\item {\rm{(\bf Existence of quasi-equilibrium})}\\
There is a unique solution $(\blambda^R,\bF^R)$ to \eqref{E-nonlinear-eigenproblem-truncated}. In addition, $\bF^R=e^{-\bV^R}\in L^1_+(\mathbb{R})\cap C^\infty(\bar I_R)$ is compactly supported on $\bar I_R$ and bounded away from zero on it and $\alphalpha$-log-concave with parameter $\alphalpha>0$ given in \eqref{E-main-alpha} in Theorem \ref{T-main}.
\item {\rm ({\bf One-step contraction})}\\
Consider any $F_0^R\in L^1_+(\mathbb{R})\cap C^1(\bar I_R)$ compactly supported on $\bar I_R$ and bounded away from zero on it, and let $\{F_n^R\}_{n\in \mathbb{N}}$ be the solution to \eqref{E-nonlinear-truncated} issued at $F_0^R$. Then, we have
$$
\left\Vert\frac{d}{dx} \left(\log \frac{F_n^R}{\bF^R}\right)\right\Vert_{L^\infty(\bar I_R)}\leq \frac{2}{1+2\alphalpha}\,\left\Vert \frac{d}{dx} \left(\log \frac{F_{n-1}^R}{\bF^R}\right)\right\Vert_{L^\infty(\bar I_R)},
$$
for any $n\in \mathbb{N}$.
\item {\rm ({\bf Asynchronous exponential growth})}\\
Consider any $F_0^R\in L^1_+(\mathbb{R})\cap C^1(\bar I_R)$ compactly supported on $\bar I_R$ and bounded away from zero on it, and let $\{F_n^R\}_{n\in \mathbb{N}}$ be the solution to \eqref{E-nonlinear-truncated} issued at $F_0^R$. Then, we have
\begin{align*}
\left\vert\frac{\Vert F_n^R\Vert_{L^1}}{\Vert F_{n-1}^R\Vert_{L^1}}-\blambda^R\right\vert&\leq C_R\left(\frac{2}{1+2\alphalpha}\right)^n,\\
\left\Vert\frac{F_n^R}{\Vert F_n^R\Vert_{L^1}}-\bF^R\right\Vert_{C^1}&\leq C_R'\left(\frac{2}{1+2\alphalpha}\right)^n,
\end{align*}
for any $n\in \mathbb{N}$ and some constants $C_R,C_R'$ depending on $R$ and $F_0^R$.
\end{enumerate}
\end{theorem}

As we show below, our proof exploits the overarching local contraction Lemma \ref{L-maximum-principle-l1} to answer simultaneously both questions. More specifically, our main observation is the following type of contraction which holds true providing that the initial data $F_0^R$ is strongly log-concave.

\begin{lemma}[Cauchy-type property]\label{L-Cauchy-condition} 
Let $m\in C^2(\mathbb{R})$ satisfy \eqref{H-main-m-1}-\eqref{H-main-m-2} in Theorem \ref{T-main}. Consider a $\alphalpha_0$ log-concave density $F_0^R\in L^1_+(\mathbb{R})\cap C^{1,\delta}(\bar I_R)$ with $\alphalpha_0>0$ and $0<\delta<1$, compactly supported on $\bar I_R$ and bounded away from zero on it. Let $\{F_n^R\}_{n\in \mathbb{N}}$ be the solution to \eqref{E-nonlinear-truncated} issued at $F_0^R$. Then, we have
$$\left\Vert\frac{d}{dx} \left(\log \frac{F_n^R}{F_{n-1}^R}\right)\right\Vert_{L^\infty(\bar I_R)}\leq \frac{2}{1+2\alphalpha_{n-2}}\left\Vert \frac{d}{dx}  \left(\log \frac{F_{n-1}^R}{F_{n-2}^R}\right)\right\Vert_{L^\infty(\bar I_R)},\quad n\geq 2,$$
where the sequence $\{\alphalpha_n\}_{n\in \mathbb{N}}$ is defined by recurrence like in \eqref{E-alphan}.
\end{lemma}

\begin{proof}
For any $n\in \mathbb{N}$, we define
$$u_n^R(x):=\frac{F_n^R(x)}{F_{n-1}^R(x)},\quad x\in \bar I_R,$$
and note that, arguing as in \eqref{E-nonlinear-normalization}, we have that $\{u_n\}_{n\in \mathbb{N}}$ must solve the following analogue of \eqref{E-nonlinear-normalized}:
$$u_n^R(x)=\frac{\Vert F_{n-2}^R\Vert_{L^1}}{\Vert F_{n-1}^R\Vert_{L^1}}\iint_{\bar Q_R}P_n^R(x;x_1,x_2)\,u_{n-1}^R(x_1)\,u_{n-1}^R(x_2)\,dx_1\,dx_2,$$
for any $x\in \bar I_R$ and $n\geq 2$. We remark that the system above holds only on $\bar I_R$ and the one-step transition probability $P_n^R(x;\cdot)\in L^1_+(\bar Q_R)\cap \mathcal{P}(\bar Q_R)$ is not time-homogeneous but it depends explicitly on $n$, namely
\begin{align*}
&P_n^R(x;x_1,x_2):=\frac{1}{Z_n^R(x)}e^{-W_n^R(x;x_1,x_2)},\quad x\in \bar I_R,\quad (x_1,x_2)\in \bar Q_R,\\
&W_n^R(x;x_1,x_2):=\frac{1}{2}\left\vert x-\frac{x_1+x_2}{2}\right\vert^2+V_{n-2}^R(x_1)+V_{n-2}^R(x_2),\\
&Z_n^R(x):=\iint_{\bar Q_R}e^{-W_n^R(x;x_1,x_2)}\,dx_1\,dx_2,
\end{align*}
where we denote $V_n^R:\bar I_R\longrightarrow \mathbb{R}$ so that $F_n^R=e^{-V_n^R}$. By Lemma \ref{L-stability-log-concavity} we have that $V_{n-2}^R$ is $\alphalpha_{n-2}$-convex and therefore the contractivity Lemma \ref{L-maximum-principle-l1} applies to $P_n^R(x;\cdot)$ and $P_n^R(\tilde x;\cdot)$ with $x,\tilde x\in \bar I_R$ leading to
$$W_{\infty,1}(P_n^R(x;\cdot),P_n^R(\tilde x;\cdot))\leq \frac{2}{1+2\alphalpha_{n-2}}|x-\tilde x|.$$
Therefore, arguing as in Lemma \ref{L-log-estimate} we end the proof.
\end{proof}

\begin{proof}[Proof of Theorem \ref{T-main-truncated}]

~
\smallskip

$\diamond$ {\sc Step 1}: Proof of $(i)$.\\
Under appropriate assumptions on $F_0^R$ we shall prove that $F_n^R/\Vert F_n^R\Vert_{L^1}$ and $\Vert F_n^R\Vert_{L^1}/\Vert F_{n-1}^R\Vert_{L^1}$ must converge as in $(iii)$, and their limit $(\blambda^R,\bF^R)$ solves \eqref{E-nonlinear-eigenproblem-truncated}. We set a $\alphalpha_0$-log-concave density $F_0^R\in L^1_+(\mathbb{R})\cap C^{1,\delta}(\bar I_R)$ with $\alphalpha_0>\alphalpha$ and $0<\delta<1$, compactly supported on $\bar I_R$ and bounded away from zero on it. Let $\{F_n^R\}_{n\in \mathbb{N}}$ be the solution to \eqref{E-nonlinear-truncated}. Since the initial datum has been chosen strongly log-concave, Lemma \ref{L-Cauchy-condition} implies
$$\left\Vert \frac{d}{dx}\left(\log \frac{F_n^R}{F_{n-1}^R}\right)\right\Vert_{L^\infty(\bar I_R)}\leq \left(\frac{2}{1+2\alphalpha}\right)^{n-1} \left\Vert\frac{d}{dx} \left(\log \frac{F_1^R}{F_{0}^R}\right)\right\Vert_{L^\infty(\bar I_R)},$$
for all $n\geq 1$ because $F_n^R$ are $\alphalpha_n$-log-concave with $\alphalpha_n>\alphalpha$ for all $n\in \mathbb{N}$ by Lemma \ref{L-stability-log-concavity}. Setting $V_n^R:\bar I_R\longrightarrow \mathbb{R}$ as before so that $F_n^R=e^{-V_n^R}$ we obtain
$$\left \Vert \frac{d}{dx} (V_n^R-V_m^R)\right \Vert_{L^\infty(\bar I_R)}\leq \sum_{k=m+1}^{n}\left \Vert \frac{d}{dx} (V_k^R-V_{k-1}^R)\right \Vert_{L^\infty(\bar I_R)}\leq \sum_{k=m}^{n-1}\left(\frac{2}{1+2\alphalpha}\right)^k\left \Vert \frac{d}{dx} (V_1^R-V_0^R)\right \Vert_{L^\infty(\bar I_R)},$$
for all $n\geq m\geq 1$. Since $\frac{2}{1+2\alphalpha}<1$ by Remark \ref{R-parameters-behavior}, then $\left \{\frac{d}{dx}(V_n^R)\right \}_{n\in \mathbb{N}}$ is a Cauchy sequence in $C(\bar I_R)$ and therefore it must converge uniformly to some limit $D^R\in C(\bar I_R)$. In particular, we have
\begin{equation}\label{E-convergence-log-derivative}
\frac{d}{dx}\left(\log  F_n^R \right)\rightarrow D^R\quad \mbox{in}\quad C(\bar I_R).
\end{equation}
Now, we show that $F_n^R/\Vert F_n^R\Vert_{L^1}$ must also converge when evaluated at least at one point, and we choose $x=0$ for instance. To this purpose, we note that $F_n^R(0)/\Vert F_n^R\Vert_{L^1}$ can be restated as follows
$$
\frac{\displaystyle\iint_{\bar Q_R} G\left(\frac{x_1+x_2}{2}\right)\,\exp\left(-(V_{n-1}^R(x_1)-V_{n-1}^R(0))-(V_{n-1}^R(x_2)-V_{n-1}^R(0))\right)dx_1\,dx_1}{\displaystyle\int_{\bar I_R} \iint_{\bar Q_R} G\left(x'-\frac{x_1+x_2}{2}\right)\,\exp\left(-m(x')-(V_{n-1}^R(x_1)-V_{n-1}^R(0))-(V_{n-1}^R(x_2)-V_{n-1}^R(0))\right)dx'\,dx_1\,dx_2},
$$
and $V_{n-1}^R(x)-V_{n-1}^R(0)$ in the integrand can be represented by the fundamental theory of calculus by
$$V_{n-1}^R(x)-V_{n-1}^R(0)=\int_0^1 \frac{d V_{n-1}^R}{dx}(\theta x)\,x\,d\theta,\quad x\in \bar I_R,$$
which converges uniformly to some limit. Therefore, there exists $L^R\in \mathbb{R}$ such that
\begin{equation}\label{E-convergence-log-one-point}
\log \frac{F_n^R(0)}{\Vert F_n^R\Vert_{L^1}}\rightarrow L^R.
\end{equation}
Putting \eqref{E-convergence-log-derivative}-\eqref{E-convergence-log-one-point} together and using the fundamental theorem of calculus entail
$$
\log \frac{F_n^R(x)}{\Vert F_n^R\Vert_{L^1}}=\log \frac{F_n^R(0)}{\Vert F_n^R\Vert_{L^1}}+\int_0^1\frac{d}{dx}\left(\log  F_n^R \right)(\theta x)\,x\,d\theta\rightarrow L^R+\int_0^1 D^R(\theta x)\,x\,d\theta\quad \mbox{in}\quad C^1(\bar I_R).
$$
We define $\bF^R(x):=\exp(L^R+\int_0^1 D^R(\theta x)\,x\,d\theta+\chi_{\bar I_R}(x))\in L^1_+(\mathbb{R})\cap \mathcal{P}(\mathbb{R})$ and therefore we achieve
\begin{equation}\label{E-convergence-normalized-profiles-truncated}
\frac{F_n^R}{\Vert F_n^R\Vert_{L^1}}\rightarrow \bF^R\quad \mbox{in}\quad C^1(\bar I_R).
\end{equation}

Our second step is to prove the convergence of $\Vert F_n^R\Vert_{L^1}/\Vert F_{n-1}^R\Vert_{L^1}$. Note that we have
\begin{equation}\label{E-convergence-rate-growth-truncated-pre}
\frac{\Vert F_n^R\Vert_{L^1}}{\Vert F_{n-1}^R\Vert_{L^1}}=\iint_{\mathbb{R}^2} H_R(x_1,x_2)\,\frac{F_{n-1}^R(x_1)}{\Vert F_{n-1}^R\Vert_{L^1}}\frac{F_{n-1}^R(x_2)}{\Vert F_{n-1}^R\Vert_{L^1}}\,dx_1\,dx_2,
\end{equation}
where we have defined
$$H_R(x_1,x_2):=\int_{\bar I_R} e^{-m(x)}G\left(x-\frac{x_1+x_2}{2}\right)\,dx,\quad (x_1,x_2)\in \mathbb{R}^2.$$
Since $H_R$ is a bounded function, therefore $H_R\in L^1(\bar Q_R)$ and, consequently, the above uniform convergence \eqref{E-convergence-normalized-profiles-truncated} of the normalized profiles along with \eqref{E-convergence-rate-growth-truncated-pre} imply that there must exists $\blambda^R$ with
\begin{equation}\label{E-convergence-rate-growth-truncated}
\frac{\Vert F_n^R\Vert_{L^1}}{\Vert F_{n-1}^R\Vert_{L^1}}\rightarrow \blambda^R.
\end{equation} 

The last step is to show that $(\blambda^R,\bF^R)$ must solve \eqref{E-nonlinear-eigenproblem-truncated}. This is actually clear because we have
$$
\frac{\Vert F_n^R\Vert_{L^1}}{\Vert F_{n-1}^R\Vert_{L^1}}\frac{F_n^R}{\Vert F_n^R\Vert_{L^1}}=\mathcal{T}_R\left[\frac{F_{n-1}^R}{\Vert F_{n-1}^R\Vert_{L^1}}\right],
$$
for all $n\in \mathbb{N}$, and  $\Vert F_n^R\Vert_{L^1}/\Vert F_{n-1}^R\Vert_{L^1}$ and $F_n^R/\Vert F_n^R\Vert_{L^1}$ converge in the above sense \eqref{E-convergence-normalized-profiles-truncated}-\eqref{E-convergence-rate-growth-truncated}. We note that $\bF^R$ must be $\alphalpha$-log-concave because so is $F_n^R$ for all $n\in \mathbb{N}$. The uniqueness of solution to \eqref{E-nonlinear-eigenproblem-truncated} will not be analyzed here, but it will hold as a consequence of the next contraction property in {\sc Step 2}. 

\medskip

$\diamond$ {\sc Step 2}: Proof of $(ii)$.\\
Once a strongly log-concave solution $(\blambda^R,\bF^R)$ of the truncated nonlinear eigenproblem \eqref{E-nonlinear-eigenproblem-truncated} exists, the one-step contraction property follows the same ideas as in the global version in Theorem \ref{T-main}(ii) sketched in Section \ref{S-sketch-one-step-contraction}. More specifically, we shall argue like in the proof of Lemma \ref{L-Cauchy-condition} where again we replace $u_n$ by the normalization of $F_n^R$ by the quasi-equilibrium $(\blambda^R)^n\bF^R$. That is, for any $n\in \mathbb{N}$, we define
$$u_n^R(x):=\frac{F_n^R(x)}{(\blambda^R)^n\bF^R},\quad x\in \bar I_R,$$
which must solve
$$u_n^R(x)=\frac{1}{\Vert u_{n-1}^R \bF^R \Vert_{L^1}}\iint_{\bar Q_R}P^R(x;x_1,x_2)\,u_{n-1}^R(x_1)\,u_{n-1}^R(x_2)\,dx_1\,dx_2,$$
for any $x\in \bar I_R$ and $n\in \mathbb{N}$, where $P^R(x;\cdot)\in L^1_+(\bar Q_R)\cap \mathcal{P}(\bar Q_R)$ is the one-step transition probability
\begin{align*}
&P^R(x;x_1,x_2):=\frac{1}{Z^R(x)}e^{-W^R(x;x_1,x_2)},\quad x\in \bar I_R,\quad (x_1,x_2)\in \bar Q_R,\\
&W^R(x;x_1,x_2):=\frac{1}{2}\left\vert x-\frac{x_1+x_2}{2}\right\vert^2+\bV^R(x_1)+\bV^R(x_2),\\
&Z^R(x):=\iint_{\bar Q_R}e^{-W^R(x;x_1,x_2)}\,dx_1\,dx_2.
\end{align*}
Again, we denote $\bV^R:\bar I_R\longrightarrow \mathbb{R}$ so that $\bF^R=e^{-\bV^R}$. By {\sc Step 1} we have that $\bV^R$ is $\alphalpha$-convex and therefore the contractivity Lemma \ref{L-maximum-principle-l1} applies to $P^R(x;\cdot)$ and $P^R(\tilde x;\cdot)$ with $x,\tilde x\in \bar I_R$ leading to
$$W_{\infty,1}(P^R(x;\cdot),P^R(\tilde x;\cdot))\leq \frac{2}{1+2\alphalpha}|x-\tilde x|.$$
Therefore, arguing as in Lemma \ref{L-log-estimate} we end the proof. 

In particular, the above implies that $(\blambda^R,\bF^R)$ must be the unique solution to the truncated nonlinear eigenproblem \ref{E-nonlinear-eigenproblem-truncated}. Indeed, if a second solution $(\lambda^R,F^R)$ exists, one can always define the special solution $F_n^R(x)=(\lambda^R)^n F^R(x)$ of \eqref{E-nonlinear-truncated} and therefore the above one-step contraction implies
$$\left\Vert\frac{d}{dx}\left(\log \frac{F^R}{\bF^R}\right)\right\Vert_{L^\infty(\bar I_R)}\leq \frac{2}{1+2\alphalpha}\left\Vert\frac{d}{dx}\left(\log \frac{F^R}{\bF^R}\right)\right\Vert_{L^\infty(\bar I_R)}.$$
Since $\frac{2}{1+2\alphalpha}<1$ by Remark \ref{R-parameters-behavior}, then we have $F^R=\bF^R$ (and therefore $\lambda^R=\blambda^R$) because both $F^R$ and $\bF^R$ are probability densities by definition.

\medskip

$\diamond$ {\sc Step 3}: Proof of $(iii)$.\\
We prove that the convergence in {\sc Step 1} holds for generic initial data $F_0^R\in L^1_+(\mathbb{R})\cap C^1(\bar I_R)$ compactly supported on $\bar I_R$ and bounded away from zero on it, and not necessarily strongly log-concave. Note that by the above one-step contractivity property we have again
$$\left \Vert\frac{d}{dx} (V_n^R-\bV^R)\right \Vert_{L^\infty(\bar I_R)}\leq \left(\frac{2}{1+2\alphalpha}\right)^n\left \Vert \frac{d}{dx}(V_0^R-\bV^R)\right \Vert_{L^\infty(\bar I_R)},$$
for all $n\in \mathbb{N}$. Then, the same argument as in {\sc Step 1} can be applied with explicit convergence rates and equal to $\left(\frac{2}{1+2\alphalpha}\right)^n$ at each step: first $\frac{d}{dx}(\log F_n^R)$, second $\log\left ( F_n^R(0)/\Vert F_n^R\Vert_{L^1}\right )$, hence $\log \left (F_n^R/\Vert F_n^R\Vert_{L^1}\right )$, and finally also $\Vert F_n^R\Vert_{L^1}/\Vert F_{n-1}^R\Vert_{L^1}$. Therefore, we readily obtain the claimed convergence rates for the rates of growth and the normalized profiles.
\end{proof}

\section{Existence and uniqueness of strongly log-concave quasi-equilibria}\label{S-equilibria}

In this section, we employ the truncated quasi-equilibria in the above Theorem \ref{T-main-truncated} to build a globally defined quasi-equilibrium of the non-truncated model \eqref{E-nonlinear}, thus proving Theorem \ref{T-main}(i). In the following, we show that the family of probability densities $\{\bF^R\}_{R>0}$ are uniformly tight, and therefore weak limits cannot loose mass at infinity, which will be useful in the sequel in order to pass to the limit with $R\rightarrow \infty$.

\begin{proposition}[Bounded second-order moments]\label{P-properties-truncated-equilibria}
Under the assumptions in Theorem \ref{T-main-truncated}, let us consider the unique eigenpair $(\blambda^R,\bF^R)$ of \eqref{E-nonlinear-eigenproblem-truncated} for any $R>0$ according to Theorem \ref{T-main-truncated}(i). Then, 
\begin{equation}\label{E-bounded-second-order-moments}
\sup_{R>0}\,\int_{\mathbb{R}}x^2 \bF^R(x)\,dx<\infty.
\end{equation}
\end{proposition}

We recall that a similar result was necessary in \cite{CLP-21-arxiv}. Indeed, a general strategy was developed therein to propagate second-order moments along any solution $\{F_n\}_{n\in \mathbb{N}}$ under the {\it a priori} knowledge that the centers of mass stay uniformly bounded. However, such a condition proved difficult to verify unless the initial datum $F_0$ is centered at the origin, and $m$ is an even function, which would leave the center of mass fixed at the origin (and thus bounded) for all times. To overcome this problem, an alternative approach was developed in \cite[Lemma 4.5]{CLP-21-arxiv} in order to control the convergence to zero of the center of mass in the case of quadratic selection. Unfortunately, the proof exploits the Gaussian structure in a crucial way and cannot be easily adapted to more general selection functions. Here, we propose an alternative strategy based on the extra knowledge that $\bF^R$ are $\alphalpha$-log-concave.

\begin{proof}[Proof of Proposition \ref{P-properties-truncated-equilibria}]

~

\smallskip

$\diamond$ {\sc Step 1}: Uniform bound of the variance.\\
Let us define the center of mass and the variance
\begin{align*}
\bmu_R&:=\int_{\mathbb{R}} x\bF^R(x)\,dx,\\
\bsigma_R^2&:=\int_{\mathbb{R}}\left(x-\bmu_R\right)^2\bF^R(x)\,dx,
\end{align*}
for any $R>0$. Since each eigenfunction $\bF^R$ is $\alphalpha$-log-concave, then a straightforward application of the Brascamp-Lieb inequality shows that variances $\bsigma_R^2$ verify
\begin{equation}\label{E-bounded-variance}
\bsigma_R^2\leq \frac{1}{\alphalpha},
\end{equation}
for any $R>0$, see \cite[Theorem 4.1]{BL-76}. Then, in order to control the (non-centered) second order moments, we actually need to find a bound of the center of mass $\bmu_R$.

\medskip

$\diamond$ {\sc Step 2}: Uniform bound of the center of mass.\\
Assume that $\{\bmu_R\}_{R>0}$ is unbounded by contradiction. Changing variables $x$ with $-x$ if necessary, we may assume without loss of generality that $\bmu_R\nearrow +\infty$ as $R\nearrow +\infty$ up to an appropriate subsequence, which we denote in the same way for simplicity of notation. Note that integrating \eqref{E-nonlinear-eigenproblem-truncated} against $e^{m_R(x)}$ and remarking that $\int_{\mathbb{R}}\mathcal{B}[\bF^R](x)\,dx=\int_{\mathbb{R}}\bF^R(x)\,dx=1$ (where $\mathcal{B}$ is given in \eqref{E-operator-B}) we obtain
\begin{equation}\label{E-AR-BR-1}
A_R\,B_R=1,
\end{equation}
for every $R>0$, where each factor reads
\begin{align*}
A_R&:=\int_{\mathbb{R}} e^{m_R(x)}\,\bF^R(x)\,dx,\\
B_R&:=\int_{\mathbb{R}^2}\phi^R\left(\frac{x_1+x_2}{2}\right)\bF^R(x_1)\,\bF^R(x_2)\,dx_1\,dx_2,
\end{align*}
and $\phi^R:=G*e^{-m_R}$. By Chebyshev's inequality we know that
\begin{equation}\label{E-Chebyshev}
\int_{|x-\bmu_R|\leq \sqrt{2}\bsigma_R}\bF^R(x)\,dx\geq \frac{1}{2},
\end{equation}
for all $R>0$. Therefore, noting that $m$ is non-decreasing in $\mathbb{R}_+$ by virtue of the hypothesis \eqref{H-main-m-1}-\eqref{H-main-m-2} we obtain the following lower bound
\begin{align}\label{E-AR-bound}
\begin{aligned}
A_R&\geq \int_{|x-\bmu_R|\leq \sqrt{2}\bsigma_R}e^{m_R(x)}\,\bF^R(x)\,dx\\
&\geq \frac{1}{2}\min_{|x-\bmu_R|\leq \sqrt{2}\bsigma_R} e^{m(x)}=\frac{1}{2}e^{m(\bmu_R-\sqrt{2}\bsigma_R)},
\end{aligned}
\end{align}
for large enough $R>0$ so that $[\bmu_R-\sqrt{2}\bsigma_R,\bmu_R+\sqrt{2}\bsigma_R]\subset \mathbb{R}_+$. Similarly, using \eqref{E-Chebyshev} and noting that $\phi^R$ is non-increasing at the right of its maximizer (by strong log-concavity, {\it cf}. Lemma \ref{L-stability-log-concavity}) we obtain
\begin{align}\label{E-BR-bound}
\begin{aligned}
B_R&\geq \iint_{|x_i-\bmu_R|\leq \sqrt{2}\sigma_R}\phi^R\left(\frac{x_1+x_2}{2}\right)\,\bF^R(x_1)\,\bF^R(x_2)\,dx_1\,dx_2\\
&\geq \frac{1}{4}\min_{|x-\bmu_R|\leq \sqrt{2}\bsigma_R}\phi^R(x)\geq \frac{1}{4}\phi^R(\bmu_R+\sqrt{2}\bsigma_R),
\end{aligned}
\end{align}
for large enough $R>0$ so that $[\bmu_R-\sqrt{2}\bsigma_R,\bmu_R+\sqrt{2}\bsigma_R]$ lies in that region of the domain. Note that the above can be obtained if $R>0$ is large enough since $\bmu^R-\sqrt{2}\bsigma_R\rightarrow \infty$ by assumptions, but however the maximizers of $\phi^R$ must converge to the maximizer of $\phi$, which is a fixed number in the real line. Multiplying \eqref{E-AR-bound} and \eqref{E-BR-bound} yields the lower bound
\begin{equation}\label{E-AR-BR-bound-pre}
A_RB_R\geq \frac{1}{8} e^{m_R(\bmu_R-\sqrt{2}\bsigma_R)}\,(G*e^{-m_R})(\bmu_R+\sqrt{2}\bsigma_R),
\end{equation}
for large enough $R>0$. Lemma \ref{L-lower-bound-convolution-2} provides a explicit lower bound \eqref{E-lower-bound-convolution-2} on Gaussian convolutions. Therefore, applying it to the second factor in \eqref{E-AR-BR-bound-pre} with the choices
$$f=e^{-m},\quad \gamma=\betabeta,\quad x_0=\bmu_R,\quad \delta=\sqrt{2}\bsigma_R,$$
implies the following lower bound
\begin{align}\label{E-AR-BR-bound}
\begin{aligned}
A_RB_R&\geq G(2\sqrt{2}\bsigma_R)\int_0^{\frac{\betabeta}{\betabeta+1}\bmu_R-\frac{\sqrt{2}\bsigma_R}{\betabeta+1}}\exp\left(\frac{\betabeta+1}{2}z^2\right)\,dz\\
&\geq G\left(\frac{2\sqrt{2}}{\sqrt{\alphalpha}}\right)\int_0^{\frac{\betabeta}{\betabeta+1}\bmu_R-\frac{\sqrt{2}}{\sqrt{\alphalpha}(\betabeta+1)}}\exp\left(\frac{\betabeta+1}{2}z^2\right)\,dz,
\end{aligned}
\end{align}
where in the last line we have used the bound \eqref{E-bounded-variance} of variances. Since the left hand side in \eqref{E-AR-BR-bound} diverges as $R\rightarrow \infty$ because $\bmu_R\rightarrow +\infty$, then we reach a contradiction with \eqref{E-AR-BR-1}, and this ends the proof.
\end{proof}

\begin{theorem}[Existence of quasi-equilibria]
Under the assumptions in Theorem \ref{T-main-truncated}, let us consider the unique eigenpair $(\blambda^R,\bF^R)$ of \eqref{E-nonlinear-eigenproblem-truncated} for any $R>0$. Then, there exist $\blambda\in \mathbb{R}$ and $\bF\in L^1_+(\mathbb{R})\cap C^\infty(\mathbb{R})$ which is $\alphalpha$-log-concave (with $\alphalpha$ given in \eqref{E-main-alpha}) such that
$$\blambda^R\rightarrow \blambda,\quad \bF^R\rightarrow \bF,\quad \mbox{as}\quad R\rightarrow\infty,$$
up to subsequence, both pointwise and in any space $(\mathcal{P}_p(\mathbb{R}),W_p)$ with $1\leq p<2$. Moreover, the pair $(\blambda,\bF)$ is the unique solution to \eqref{E-nonlinear-eigenproblem} among all pairs $(\lambda,F)$ verifying \eqref{E-main-F-uniqueness}.
\end{theorem}

\begin{proof}

~

\smallskip

$\diamond$ {\sc Step 1}: Existence via limit as $R\rightarrow \infty$.\\ 
Let us notice that by \eqref{E-bounded-second-order-moments} in Proposition \ref{P-properties-truncated-equilibria} we have that $\{\bF^R\}_{R>0}$ is a uniformly tight sequence of probability measures. Therefore, by Prokhorov's theorem there must exist $R_n\nearrow \infty$ and some limiting probability measure $\bF\in \mathcal{P}(\mathbb{R})$ such that 
\begin{equation}\label{E-convergence-truncations-2}
\bF^{R_n}\rightarrow \bF\quad \mbox{narrowly in }\mathcal{P}(\mathbb{R}).
\end{equation}
By integration on \eqref{E-nonlinear-eigenproblem-truncated} we also obtain that
$$\blambda^{R_n}=\iint_{\mathbb{R}^2} (e^{-m_{R_n}}*G)\left(\frac{x_1+x_2}{2}\right)\,\bF^{R_n}(x_1)\,\bF^{R_n}(x_2)\,dx_1\,dx_2,$$
and then we can pass to the limit as $n\rightarrow \infty$ in the eigenvalues too. Specifically, since $e^{-m_R}\rightarrow e^{-m}$ in $L^\infty(\mathbb{R})$, then $e^{-m_R}*G\rightarrow e^{-m}*G$ in $C_b(\mathbb{R})$, and therefore by \eqref{E-convergence-truncations-2} we obtain
\begin{equation}\label{E-convergence-truncations-3}
\blambda^{R_n}\rightarrow\blambda,
\end{equation}
as $n\rightarrow\infty$, where $\blambda$ is given by
\begin{equation}\label{E-convergence-truncation-identify-lambda}
\blambda:=\iint_{\mathbb{R}^2} (e^{-m}*G)\left(\frac{x_1+x_2}{2}\right)\,\bF(x_1)\,\bF(x_2)\,dx_1\,dx_2=\int_{\mathbb{R}} \mathcal{T}[\bF](x)\,dx.
\end{equation}
Putting \eqref{E-convergence-truncations-2} and \eqref{E-convergence-truncations-3} together and taking limits as $n\rightarrow \infty$ in \eqref{E-nonlinear-eigenproblem-truncated} implies that $\{\bF^{R_n}\}_{n\in \mathbb{N}}$ must also converge pointwise to some other limit $\tilde \bF\in L^1_+(\mathbb{R})$ by Fatou's lemma. Note that since $\bF^R$ are all $\alphalpha$-log-concave, then so must also be their pointwise limit $\tilde \bF$. Indeed, note that we further have
\begin{equation}\label{E-convergence-truncation-identify-equation}
\blambda\tilde \bF(x)=\mathcal{T}[\bF](x),\quad x\in \mathbb{R},
\end{equation}
and therefore, $\tilde \bF\in L^1_+(\mathbb{R})\cap\mathcal{P}(\mathbb{R})$, in view of \eqref{E-convergence-truncation-identify-lambda}. Then, we actually have $\bF^{R_n}\rightarrow \tilde \bF$ in $L^1(\mathbb{R})$ (thus narrowly in $\mathcal{P}(\mathbb{R})$) by Scheff\'e's lemma. Since $\bF$ is a narrow limit of the same sequence, then we have $\tilde \bF=\bF$ and by \eqref{E-convergence-truncation-identify-equation} we obtain that $(\blambda,\bF)$ must verify the initial problem \eqref{E-nonlinear-eigenproblem}.  Let us also emphasize that, we indeed have convergence in any $L^p$ Wasserstein space with $1\leq p<2$ because all the $p$-th order moment with $1\leq p<2$ are uniformly integrable by \eqref{E-bounded-second-order-moments}, see \cite[Proposition 7.1.5]{AGS-08}.

\medskip

$\diamond$ {\sc Step 2}: Uniqueness of quasi-equilibria.\\
Note that several different convergent subsequences of $\{\bF^R\}_{R>0}$ in {\sc Step 1} could give rise to various eigenpairs $(\blambda,\bF)$ of \eqref{E-nonlinear-eigenproblem}. Whilst the global uniqueness is unclear with this method, we prove that there can only exist one solution to \eqref{E-nonlinear-eigenproblem} among the pairs $(\lambda,F)$ verifying \eqref{E-main-F-uniqueness}. For, we exploit the one-step contraction property in Theorem \ref{T-main}(ii). Specifically, assume that $(\lambda,F)$ is any other solution to \eqref{E-nonlinear-eigenproblem} and define $F_n(x)=\lambda^n F(x)$, which is clearly a solution to the evolution problem \eqref{E-nonlinear} with initial datum $F_0\in L^1_+(\mathbb{R})\cap C^1(\mathbb{R})$ verifying the hypothesis \eqref{H-main-F0} by virtue of the assumption \eqref{E-main-F-uniqueness}. Then, \eqref{E-main-one-step-contraction} implies
$$\left\Vert\frac{d}{dx} \left(\log \frac{F}{\bF}\right)\right\Vert_{L^\infty}\leq \frac{2}{1+2\alphalpha}\left\Vert \frac{d}{dx}\left(\log \frac{F}{\bF}\right)\right\Vert_{L^\infty}.$$
Again, since $\frac{2}{1+2\alphalpha}<1$ by Remark \ref{R-parameters-behavior}, then we obtain that $F/\bF$ must be constant. Since both $F$ and $\bF$ are normalized probability densities, then we necessarily have that $F=\bF$ (and therefore $\lambda=\blambda$).
\end{proof}

\section{Convergence to equilibrium for restricted initial data}\label{S-convergence-equilibrium}

In this section, we prove asynchronous exponential as claimed in Theorem \ref{T-main}(iii). More specifically, we show that for restricted initial the asymptotic behavior of the rate of growth of mass $\Vert F_n\Vert_{L^1}/\Vert F_{n-1}\Vert_{L^1}$ and the normalized profiles $F_n/\Vert F_n\Vert_{L^1}$ is dictated by the solution $(\blambda,\bF)$ of the eigenproblem \eqref{E-nonlinear-eigenproblem} obtained in Theorem \ref{T-main}(i). We derive the relaxation of the normalized profiles under the relative entropy metric. Our starting point is the one-step contraction property of the $L^\infty$ relative Fisher information in Theorem \ref{T-main}(ii) and the following version of the logarithmic-Sobolev inequality with respect to strongly log-concave densities, which relate the ($L^2$) relative Fisher information and the relative entropy.

\begin{proposition}[Logarithmic-Sobolev inequality]\label{P-LSI}
Consider any couple $P,Q\in L^1_+(\mathbb{R})\cap \mathcal{P}(\mathbb{R})$ such that $Q$ is $\gamma$-log-concave for some $\gamma>0$. Then, we have
\begin{equation}\label{E-LSI}
\mathcal{D}_{KL}(P\Vert Q)\leq \frac{1}{2\gamma}\,\mathcal{I}_2(P\Vert Q)\leq\frac{1}{2\gamma}\,\mathcal{I}_\infty^2(P\Vert Q),
\end{equation}
where $\mathcal{D}_{\rm KL}$ is the relative entropy \eqref{E-main-relative-entropy}, $\mathcal{I}_2$ is the usual (or $L^2$) relative Fisher  information \eqref{E-main-Fisher-information-2},
and $\mathcal{I}_\infty$ is the $L^\infty$ relative Fisher  information \eqref{E-main-Fisher-information-infty}.
\end{proposition}

On the one hand, the first part of the inequality \eqref{E-LSI} amounts to the usual logarithmic-Sobolev inequality with respect to a strongly log-concave measure, see Corollary 5.7.2 and Section 9.3.1 in \cite{BGL-14} for details. On the other hand, the second part of the inequality readily holds by definition. Therefore, putting Theorem \ref{T-main}(ii) and Proposition \eqref{P-LSI} together, we end the proof of Theorem \ref{T-main}(iii).

\begin{proof}[Proof of Theorem \ref{T-main}(iii)]
Notice that by iterating $n$ times the one-step contraction property in Theorem \ref{T-main}(ii) and using the logarithmic-Sobolev inequality \eqref{E-LSI} in Proposition \ref{P-LSI} we obtain
\begin{equation}\label{E-relaxation-normalized-profiles}
\mathcal{D}_{\rm KL}\left(\left.\frac{F_n}{\Vert F_n\Vert_{L^1}}\right\Vert\bF\right)\leq C_1\,\left(\frac{2}{1+2\alphalpha}\right)^{2n},
\end{equation}
for every $n\in \mathbb{N}$, where the constant $C_1$ reads
$$C_1:=\frac{1}{2\gamma}\,\mathcal{I}_\infty^2\left(\left. F_0\right\Vert\bF\right),$$
and it is finite by the assumption \eqref{H-main-F0}. This proves the relaxation of the normalized profiles towards $\bF$ in the relative entropy sense. Regarding the rate of growth, we note that
\begin{align}
\frac{\Vert F_n\Vert_{L^1}}{\Vert F_{n-1}\Vert_{L^1}}&=\iint_{\mathbb{R}^2}\phi\left(\frac{x_1+x_2}{2}\right)\frac{F_{n-1}(x_1)}{\Vert F_{n-1}\Vert_{L^1}}\frac{F_{n-1}(x_2)}{\Vert F_{n-1}\Vert_{L^1}}\,dx_1\,dx_2,\label{E-rate-growth-n-explicit}\\
\blambda&=\iint_{\mathbb{R}^2}\phi\left(\frac{x_1+x_2}{2}\right)\bF(x_1)\,\bF(x_2)\,dx_1\,dx_2.\label{E-rate-growth-limit-explicit}
\end{align}
where $(\blambda,\bF)$ is the solution to \eqref{E-nonlinear-eigenproblem} in Theorem \ref{T-main}(i), and $\phi:=G*e^{-m}$ again. Taking the difference of the two identities \eqref{E-rate-growth-n-explicit} and \eqref{E-rate-growth-limit-explicit} above, we achieve
\begin{align*}
\left\vert\frac{\Vert F_n\Vert_{L^1}}{\Vert F_{n-1}\Vert_{L^1}}-\blambda\right\vert & \leq \Vert \phi\Vert_{L^\infty} \left\Vert \frac{F_{n-1}}{\Vert F_{n-1}\Vert_{L^1}}\otimes\frac{F_{n-1}}{\Vert F_{n-1}\Vert_{L^1}} -\bF\otimes \bF\right\Vert_{L^1}\\
& \leq \Vert \phi\Vert_{L^\infty}\sqrt{\frac{1}{2}\mathcal{D}_{\rm KL}\left(\left.\frac{F_{n-1}}{\Vert F_{n-1}\Vert_{L^1}}\otimes\frac{F_{n-1}}{\Vert F_{n-1}\Vert_{L^1}}\right\Vert\bF\otimes \bF\right)}\\
& = \Vert \phi\Vert_{L^\infty}\sqrt{\mathcal{D}_{\rm KL}\left(\left.\frac{F_{n-1}}{\Vert F_{n-1}\Vert_{L^1}}\right\Vert\bF\right)}\\
&\leq C_2\left(\frac{2}{1+2\alphalpha}\right)^n,
\end{align*}
with a explicit constant $C_2>0$ taking the form
$$C_2:=\Vert \phi\Vert_{L^\infty}\sqrt{C_1}.$$
Note that above, we have used successively H\"{o}lder's inequality, Pinsker's inequality, the tensorization property of the relative entropy, and \eqref{E-relaxation-normalized-profiles} to reach the conclusion.
\end{proof}

%%%

\appendix

\section{Intermediate dualities}\label{A-intermediate-dualities}

For simplicity of the discussion, we do not present here the intermediate Kantorovich-type dualities in the case of non-linear transition semigroups like in \eqref{E-nonlinear-normalized}, but we rather focus on linear semigroups. More specifically, we have the following intermediate result which is reminiscent of the natural interpolation of Kantorovich duality for $L^1$ Wasserstein distance, and Lemma \ref{L-log-estimate} for $L^\infty$ Wassestein metric.

\begin{proposition}\label{P-nonlinear-dualities}
Consider any $\mu,\nu\in \mathcal{P}_p(\mathbb{R}^d)$ for some $1\leq p\leq \infty$, and set any function $u\in C^1(\mathbb{R}^d)$ such that $u>0$ and $\nabla(u^{1/p})\in L^\infty(\mathbb{R}^d,\mathbb{R}^d)$. Then, the following inequality holds true
$$\left |\left(\int_{\mathbb{R}^d}u(x)\,\mu(dx)\right)^{1/p}-\left(\int_{\mathbb{R}^d}u(x)\,\nu(dx)\right)^{1/p}\right|\leq \left\Vert \,\Vert\nabla (u^{1/p})\Vert_{q'}\right\Vert_{L^\infty}\,W_{p,q}(\mu,\nu),$$
for any $1\leq q\leq\infty$. Here, $W_{p,q}$ denotes the $L^p$ Wasserstein distance associated with $\ell_q$ norm of $\mathbb{R}^d$, {\it cf.} \eqref{E-Wpq-distance}, and we admit the convention that $u^{1/\infty}=\log u$ for all $u>0$.
\end{proposition}

\begin{proof}
Let us consider any constant-speed geodesic $\rho:[0,1]\longrightarrow \mathcal{P}_p(\mathbb{R}^d)$ in the Wasserstein space $(\mathcal{P}_p(\mathbb{R}^d),W_{p,q})$ joining $\mu$ to $\nu$. Specifically, $\rho$ verifies the continuity equation
\begin{align}\label{E-geodesic-Wpq-1}
\begin{aligned}
&\partial_t\rho+\divop(\rho v)=0,\quad t\in [0,1],\, x\in \mathbb{R}^d,\\
&\rho_0=\mu,\quad \rho_1=\nu,
\end{aligned}
\end{align}
in distributional sense and, in addition, we have
\begin{equation}\label{E-geodesic-Wpq-2}
\Vert \,\Vert v_t\Vert_q\Vert_{L^p(\rho_t)}=W_{p,q}(\mu,\nu),\quad t\in [0,1].
\end{equation}
Let us also define the function
$$E(t):=\int_{\mathbb{R}^d}u(y)\,\rho_t(dy),\quad t\in [0,1].$$
Since $\rho\in {\rm Lip}([0,1],\mathcal{P}_p(\mathbb{R}^d))$, then $E\in AC([0,1])$ and by the continuity equation $\eqref{E-geodesic-Wpq-1}$ we have
\begin{equation}\label{E-equation-derivative-E}
\frac{dE}{dt}(t)=\int_{\mathbb{R}^d}\nabla u(y)\cdot v_t(y)\,\rho_t(dy)=p\int_{\mathbb{R}^d}\nabla(u^{1/p})(y)\cdot v_t(y)\,u^{1/{p'}}(y)\,\rho_t(dy),
\end{equation}
for {\it a.e.} $t\in [0,1]$, where we have used the identity $\nabla u=p\,\nabla (u^{1/p})\,u^{1/{p'}}$.  Therefore, we obtain
\begin{align*}
\left |\frac{dE}{dt}(t)\right |&\leq p\,\int_{\mathbb{R}^d}\Vert \nabla (u^{1/p})(y)\Vert_{q'}\,\Vert v_t(y)\Vert_q \,u^{1/{p'}}(y)\,\rho_t(dy)\\
&\leq p\,\left\Vert \Vert \nabla (u^{1/p})\Vert_{q'}\right\Vert_{L^\infty}\int_{\mathbb{R}^d}\Vert v_t(y)\Vert_q \,u^{1/{p'}}(y)\,\rho_t(dy)\\
&\leq p\,\left\Vert \Vert \nabla (u^{1/p})\Vert_{q'}\right\Vert_{L^\infty}\,\left\Vert \,\Vert v_t\Vert_q\right\Vert_{L^p(\rho_t)}\,\Vert u^{1/p'}\Vert_{L^{p'}(\rho_t)},
\end{align*}
for {\it a.e.} $t\in [0,1]$, where in the first step we have used H\"{o}lder's inequality with exponent $q$ applied to the inner product in the integrand of \eqref{E-equation-derivative-E}, and in the last step we have used H\"{o}lder's inequality with exponent $p$ applied to the integral of the second line. Using the constant-speed condition \eqref{E-geodesic-Wpq-2} in the first factor, and $\Vert u^{1/p'}\Vert_{L^{p'}(\rho_t)}=E(t)^{1/{p'}}$ in the last one, we obtain the relation
$$\left |\frac{dE}{dt}(t)\right |\leq p\,\left\Vert \Vert \nabla (u^{1/p})\Vert_{q'}\right\Vert_{L^\infty}W_{p,q}(\mu,\nu)\,E(t)^{1/p'},$$
for {\it a.e.} $t\in [0,1]$, which amounts to
$$\left |\frac{dE^{1/p}}{dt}(t)\right|  \leq \left\Vert \Vert \nabla (u^{1/p})\Vert_{q'}\right\Vert_{L^\infty}W_{p,q}(\mu,\nu),$$
for {\it a.e.} $t\in [0,1]$. Integrating between $0$ and $1$ implies
$$\left |E(0)^{1/p}-E(1)^{1/p}\right |\leq \left\Vert \Vert \nabla (u^{1/p})\Vert_{q'}\right\Vert_{L^\infty}W_{p,q}(\mu,\nu).$$
Then, noting that $E(0)=\int_{\mathbb{R}^d}u(x)\,\mu(dx)$ and $E(1)=\int_{\mathbb{R}^d}u(x)\,\nu(dx)$ ends the proof.
\end{proof}

As a consequence, we obtain the following result, which allows identifying the Lipschitz constant of a function with the Lipschitz constant of an associated nonlinear functional over $\mathcal{P}_p(\mathbb{R}^d)$.

\begin{corollary}\label{C-nonlinear-dualities}
Consider any $1\leq p\leq \infty$, set any $v\in C^1(\mathbb{R}^d)$ with $\nabla v\in L^\infty(\mathbb{R}^d,\mathbb{R}^d)$, and assume that $v>0$ when $p<\infty$ but not necessarily when $p=\infty$. Define the functional $\Phi_{p,v}:\mathcal{P}_p(\mathbb{R}^d)\longrightarrow \mathbb{R}$ by
$$
\Phi_{p,v}[\mu]:=\left\{
\begin{array}{ll}
\displaystyle\left(\int_{\mathbb{R}^d}v(x)^p\mu(dx)\right)^{1/p}, & \mbox{if }p<\infty,\\
\displaystyle\log\left(\int_{\mathbb{R}^d}e^{v(x)}\mu(dx)\right), & \mbox{if }p=\infty,
\end{array}
\right.
$$
for any $\mu\in \mathcal{P}_p(\mathbb{R}^d)$. Then, for any $1\leq q\leq \infty$ the following identify holds true
$$\Vert \,\Vert\nabla v\Vert_{q'}\Vert_{L^\infty}=\sup_{\mu,\nu\in \mathcal{P}_p(\mathbb{R}^d)}\frac{\Phi_{p,v}[\mu]-\Phi_{p,v}[\nu]}{W_{p,q}(\mu,\nu)}.$$
\end{corollary}

\begin{proof}
First, note that the change of variables $v=u^{1/p}$ and Proposition \ref{P-nonlinear-dualities} readily implies
$$\Vert \,\Vert\nabla v\Vert_{q'}\Vert_{L^\infty}\geq \sup_{\mu,\nu\in \mathcal{P}_p(\mathbb{R}^d)}\frac{\Phi_{p,v}[\mu]-\Phi_{p,v}[\nu]}{W_{p,q}(\mu,\nu)}.$$
On the other hand, also note that by particularizing the measures $\mu,\nu\in \mathcal{P}_p(\mathbb{R}^d)$ to be Dirac masses at respective points $x,x'\in\mathbb{R}^d$ we obtain
$$\sup_{\mu,\nu\in \mathcal{P}_p(\mathbb{R}^d)}\frac{\Phi_{p,v}[\mu]-\Phi_{p,v}[\nu]}{W_{p,q}(\mu,\nu)}\geq \sup_{x,x'\in \mathbb{R}^d}\frac{\Phi_{p,v}[\delta_x]-\Phi_{p,v}[\delta_{x'}]}{W_{p,q}(\delta_x,\delta_{x'})}=\sup_{x,x'\in\mathbb{R}^d}\frac{v(x)-v(x')}{\Vert x-x'\Vert_q}=\Vert \,\Vert \nabla v\Vert_{q'}\Vert_{L^\infty}.$$
This proves the converse inequality and then the above identity holds.
\end{proof}

\section{Lower bound of Gaussian convolution of log-concave densities}\label{A-lower-bound-convolution}

We present a technical result which computes an explicit lower bound on the convolution of a Gaussian density and any strongly log-concave probability density.

\begin{lemma}[Lower bound I]\label{L-lower-bound-convolution}
Consider any $f=e^{-V}\in L^1_+(\mathbb{R})\cap \mathcal{P}(\mathbb{R})$, such that $V\in C^1(\mathbb{R})$ with $V'(0)=0$, and $f$ is $\gamma$-log-concave for some $\gamma>0$. Then, we have
\begin{equation}\label{E-lower-bound-convolution}
(G*f)(x_0+\delta)\geq G(2\delta)\,f(x_0-\delta)\,\int_0^{\frac{\gamma}{\gamma+1} x_0-\frac{\delta}{\gamma+1}} \exp\left(\frac{\gamma+1}{2} z^2\right)\,dz,
\end{equation}
for any $\delta>0$ and each $x_0>\frac{\gamma+2}{\gamma}\delta$, where $G$ denotes the standard Gaussian distribution \eqref{E-G}.
\end{lemma}

\begin{proof}
For simplicity of notation, we define $x_{\pm}:=x_0\pm \delta$ and we note that we can write
\begin{equation}\label{E-U-lower-bound-0}
(G*f)(x_+)=\frac{1}{(2\pi)^{1/2}}\,f(x_-)\,\int_\mathbb{R}e^{V(x_-)-U(x)}\,dx,
\end{equation}
where the function $U:\mathbb{R}\longrightarrow\mathbb{R}$ is defined by
$$U(x):=V(x)+\frac{1}{2}(x-x_+)^2,\quad x\in \mathbb{R}.$$
Since the potential $V$ is $\gamma$ convex, then we have that the potential $U$ is $(\gamma+1)$-convex. By the convexity inequality applied to the pair of points $(x,x_-)$ we then obtain
\begin{equation}\label{E-U-lower-bound-1}
U(x_-)\geq U(x)+U'(x)(x_--x)+\frac{\gamma+1}{2}(x_--x)^2,
\end{equation}
for any $x\in \mathbb{R}$. Consider the unique minimizer $x_*\in \mathbb{R}$ of the potential $U$. Since in particular $x_*$ is a critical point of $U$, then we have
$$0=U'(x_*)=V'(x_*)+(x_*-x_+).$$
Multiplying above by $x^*$, using that $V'(0)=0$ by hypothesis along with the convexity inequality of $V$ applied at the pair $(x_*,0)$, we infer $\gamma\,x_*^2\leq  (x_+-x_*)x_*$, and therefore,
\begin{equation}\label{E-x*-x-}
|x_*|\leq \frac{1}{\gamma+1}x_+.
\end{equation}
Since $U'(x)>0$ for $x>x_*$ and $x_--x>0$ for $x<x_-$, then \eqref{E-U-lower-bound-1} implies
$$
U(x_-)\geq U(x)+\frac{\gamma+1}{2}(x_--x)^2,
$$
for any $x\in (x_*,x_-)$. Let us note that indeed we have the appropriate ordering $x_*<x_-$ since by \eqref{E-x*-x-} and the assumption $x_0>\frac{\gamma+2}{\gamma}\delta$ we obtain
$$x_*\leq \frac{1}{\gamma+1}x_+=\frac{1}{\gamma+1}(x_0+\delta)\leq x_0-\delta=x_-.$$
Writing everything in terms of $V$ implies
\begin{equation}\label{E-U-lower-bound-2}
V(x_-)-U(x)\geq -\frac{1}{2}(x_--x_+)^2+\frac{\gamma+1}{2}(x_--x)^2,
\end{equation}
for any $x\in (x_*,x_-)$. Injecting \eqref{E-U-lower-bound-2} into \eqref{E-U-lower-bound-0} we obtain
$$(G*f)(x_+)\geq G(x_+-x_-)\,f(x_-)\,\int_{x_*}^{x_-}\exp\left(\frac{\gamma+1}{2}(x_--x)^2\right)\,dx.$$
Of course, the above implies \eqref{E-lower-bound-convolution} by a simple change of variables $z=x_--x$, and noting again that
$$x_--x_*\geq x_--\frac{1}{\gamma+1}x_+=(x_0-\delta)-\frac{1}{\gamma+1}(x_0+\delta)=\frac{\gamma}{\gamma+1}x_0-\frac{\gamma+2}{\gamma+1}\delta,$$
thanks to \eqref{E-x*-x-}, which yields again positive a positive upper bound by the assumption $x_0>\frac{\gamma+2}{\gamma}\delta$.
\end{proof}

Note that arguing along the same lines, we can prove an analogous result where the above positive strongly log-concave density $f$ is replaced by its truncation $f_R$ to intervals $I_R:=(-R,R)$. Specifically, anything that we need to guarantee is that $[x_*,x_-]\subset I_R$. First, note that $x_-<R$ amounts to the condition $x_0<R+\delta$. Second, by \eqref{E-x*-x-} we obtain that $x_*>-R$ as long as $\frac{1}{\gamma+1}x_+<R$, which amounts to the condition $x_0<(\gamma+1)R-\delta$. If we take $R$ large enough (namely $R>2\delta/\gamma$) then we have that the former condition on $x_0$ is the most restrictive. Therefore, we have the following result. 

\begin{lemma}[Lower bound II]\label{L-lower-bound-convolution-2}
Under the assumptions in Lemma \ref{L-lower-bound-convolution}, let us define
\begin{align*}
\begin{aligned}
f_R(x)&:=e^{-V_R(x)}, && x\in \mathbb{R}\\
V_R(x)&:=V(x)+\chi_{\bar I_R}(x), && x\in \mathbb{R},
\end{aligned}
\end{align*}
for any $R>0$. Then, we have
\begin{equation}\label{E-lower-bound-convolution-2}
(G*f_R)(x_0+\delta)\geq G(2\delta)\,f_R(x_0-\delta)\,\int_0^{\frac{\gamma}{\gamma+1} x_0-\frac{\delta}{\gamma+1}} \exp\left(\frac{\gamma+1}{2} z^2\right)\,dz,
\end{equation}
for any $\delta>0$, each $\frac{\gamma+2}{\gamma}\delta<x_0<R+\delta$, and every $R>\frac{2\delta}{\gamma}$.
\end{lemma}

\section{Lemmas \ref{L-properties-maximizers-l1} and \ref{L-maximum-principle-l1} for $\ell_1$ balls}\label{A-also-l1-balls}

\begin{lemma}
Lemmas \ref{L-properties-maximizers-l1} and \ref{L-maximum-principle-l1} remain true with $\ell_\infty$ balls replaced by $\ell_1$ balls.
\end{lemma}

Since the version of Lemma \ref{L-maximum-principle-l1} over $\ell_1$ balls follows the same train of thoughts as the original version for $\ell_\infty$ balls, then we just focus on proving the alternative version of Lemma \ref{L-properties-maximizers-l1} over $\ell_1$ balls. Namely, in the original argument we shall rather set $Q_R=\{z\in \mathbb{R}^2:\,\Vert z\Vert_1<R\}$.

\begin{proof}
We remark that $z^*\in \bar Q_R$ must also be a maximizer of $\tilde H$ since we have
$$\tilde H(z)\leq H(z)\leq H(z^*)=\tilde H(z^*),$$
for every $z^*\in \bar Q_R$ by the definition of $H$ and $\tilde H$ in \eqref{E-displacement-H-local} and \eqref{E-displacement-H-auxiliary-local}. Since the maximizer $z^*$ may lie in principle in all $\bar Q_R$, two possible options arise, either $z^*\in Q_R$ or $z^*\in \partial Q_R$. In the first case, the usual optimality conditions at interior points yield \eqref{E-optimality-conditions-H-tilde-local}. In the second case, namely $z^*\in \partial Q_R$, note that the result is trivial if $z^*$ is one of the four corners since those are fixed points of $T$ and therefore $\tilde{H}\equiv 0$. Hence, here on we will assume that $z^*\in \partial Q_R$ is not at a corner, but it lies in the interior of some of the four segments. Note that at those points we only have to prove that $\nabla\tilde{H}(z^*)=0$. In fact, we remark that those $z^*$ can be approached by interior points from any direction, and then the above readily implies the second order optimality condition $D^2 \tilde{H}(z^*)\leq 0$. To prove that $\nabla\tilde{H}(z^*)=0$, note that the boundary $\partial Q_R$ contains four segments:
\begin{align*}
S_{++}&:=\{(x_1,x_2)\in \mathbb{R}^2:\,x_1\geq 0,\,x_2\geq 0,\,x_1+x_2=R\},\\
S_{--}&:=\{(x_1,x_2)\in \mathbb{R}^2:\,x_1\leq 0,\,x_2\leq  0,\,-x_1-x_2=R\},\\
S_{+-}&:=\{(x_1,x_2)\in \mathbb{R}^2:\,x_1\geq 0,\,x_2\leq  0,\,x_1-x_2=R\},\\
S_{-+}&:=\{(x_1,x_2)\in \mathbb{R}^2:\,x_1\leq 0,\,x_2\geq  0,\,-x_1+x_2=R\}.
\end{align*}
Since $T(\partial Q_R)=\partial Q_R$ and each segment is mapped to itself, then we have the following information
\begin{align}
\partial_{x_1}\psi(z)+\partial_{x_2}\psi(z)&=0,\qquad \mbox{if }z\in S_{++}\cup S_{--},\label{E-transport-segments-1-bis}\\
\partial_{x_1}\psi(z)-\partial_{x_2}\psi(z)&=0,\qquad \mbox{if }z\in S_{+-}\cup S_{-+}.\label{E-transport-segments-2-bis}
\end{align}
Now, we argue according to the four possible choices for the signs of $\partial_{x_1}\psi(z^*)$ and $\partial_{x_2}\psi(z^*)$. 

\medskip

$\diamond$ {\sc Case 1}: $\partial_{x_1}\psi(z^*)\geq 0$ and $\partial_{x_2}\psi(z^*)\geq 0$.\\
In this case we have $\tilde{H}=\partial_{x_1}\psi+\partial_{x_2}\psi$. Note that it is not possible that $z^*\in S_{++}\cup S_{--}$ because otherwise we would have $H(z^*)=\tilde{H}(z^*)=0$ by \eqref{E-transport-segments-1-bis} and this of course implies that $T(z)=z$ for all $z\in \bar Q_R$, that is, $x=\tilde x$. Then we must have necessarily $z^*\in S_{+-}\cup S_{-+}$. Assume that $z^*\in S_{+-}$ (the case $z^*\in S_{-+}$ can be handled similarly). Then, by \eqref{E-transport-segments-2-bis} we obtain
\begin{align*}
\begin{aligned}
\tilde{H}(z)&=\partial_{x_1}\psi(z)+\partial_{x_2}\psi(z), && z\in \bar Q_R,\\
\tilde{H}(z)&=2\partial_{x_1}\psi(z), && z\in S_{+-}.
\end{aligned}
\end{align*}
Since $z^*$ is a maximizer of both functions, then there exist $\lambda\geq 0$ and $\mu\in \mathbb{R}$ such that the gradient of each function at $z^*$ equals the multiples $\lambda (1,-1)$ and $\mu (1,-1)$ of the outer normal vector:
\begin{align*}
\begin{aligned}
&\partial_{x_1x_1}\psi(z^*)+\partial_{x_1x_2}\psi(z^*)=\lambda, && 2\partial_{x_1x_1}\psi(z^*)=\mu,\\
&\partial_{x_1x_2}\psi(z^*)+\partial_{x_2x_2}\psi(z^*)=-\lambda, && 2\partial_{x_1x_2}\psi(z^*)=-\mu.
\end{aligned}
\end{align*}
Then, we deduce $\lambda=0$ and therefore $\nabla \tilde{H}(z^*)=0$.

\medskip

$\diamond$ {\sc Case 2}: $\partial_{x_1}\psi(z^*)<0$ and $\partial_{x_2}\psi(z^*)<0$.\\
This case follows exactly the same argument as {\sc Case 1} with $\tilde{H}$ replaced by $\tilde{H}=-\partial_{x_1}\psi-\partial_{x_2}\psi$ and then we omit the proof.

\medskip

$\diamond$ {\sc Case 3}: $\partial_{x_1}\psi(z^*)\geq 0$ and $\partial_{x_2}\psi(z^*)<0$.\\
Now we have $\tilde{H}=\partial_{x_1}\psi-\partial_{x_2}\psi$. In this case it is not possible that $z^*\in S_{+-}\cup S_{-+}$ because otherwise we would have $H(z^*)=\tilde{H}(z^*)=0$ by \eqref{E-transport-segments-2-bis}. Then we must have necessarily $z^*\in S_{++}\cup S_{--}$. Assume that $z^*\in S_{++}$ (the case $z^*\in S_{--}$ can be handled similarly). Then by \eqref{E-transport-segments-1-bis} we have
\begin{align*}
\begin{aligned}
\tilde{H}(z)&=\partial_{x_1}\psi(z)-\partial_{x_2}\psi(z), && z\in \bar Q_R,\\
\tilde{H}(z)&=2\partial_{x_1}\psi(z), && z\in S_{++}.
\end{aligned}
\end{align*}
Since $z^*$ is a maximizer of both functions, then there exists $\lambda\geq 0$ and $\mu\in \mathbb{R}$ such that the gradient of each function at $z^*$ equals the multiples $\lambda(1,1)$ and $\mu(1,1)$ of the outer normal vector:
\begin{align*}
\begin{aligned}
&\partial_{x_1x_1}\psi(z^*)-\partial_{x_1x_2}\psi(z^*)=\lambda, && 2\partial_{x_1x_1}\psi(z^*)=\mu,\\
&\partial_{x_1x_2}\psi(z^*)-\partial_{x_2x_2}\psi(z^*)=\lambda, && 2\partial_{x_1x_2}\psi(z^*)=\mu.
\end{aligned}
\end{align*}
Again we deduce $\lambda=0$ and therefore $\nabla \tilde{H}(z^*)=0$.

\medskip

$\diamond$ {\sc Case 4}: $\partial_{x_1}\psi(z^*)<0$ and $\partial_{x_2}\psi(z^*)\geq 0$.\\
This case follows exactly the same argument as {\sc Case 3} with $\tilde{H}$ replaced by $\tilde{H}=-\partial_{x_1}\psi+\partial_{x_2}\psi$ and then we omit the proof.
\end{proof}

\bibliographystyle{amsplain} %Options: amsplain, elsarticle-harv
\bibliography{biblio}

\providecommand{\bysame}{\leavevmode\hbox to3em{\hrulefill}\thinspace}
\providecommand{\MR}{\relax\ifhmode\unskip\space\fi MR }
% \MRhref is called by the amsart/book/proc definition of \MR.
\providecommand{\MRhref}[2]{%
  \href{http://www.ams.org/mathscinet-getitem?mr=#1}{#2}
}
\providecommand{\href}[2]{#2}
\begin{thebibliography}{10}

\bibitem{AGS-08}
L.~Ambrosio, N.~Gigli, and G.~Savar\'e, \emph{Gradient flows in metric spaces
  and in the space of probability measures}, Birkh{\"a}user, Basel, 2008.

\bibitem{AMTU-01}
A.~Arnold, P.~Markowich, G.~Toscani, and A.~Unterreiter, \emph{On {Convex}
  {Sobolev} {Inequalities} and the {Rate} of {Convergence} to {Equilibrium} for
  {Fokker}-{Planck} {Type} {Equations}}, Commun. Partial. Differ. Equ.
  \textbf{8} (26), no.~1-2, 43--100.

\bibitem{B-94}
D.~Bakry, \emph{L'hypercontractivité et son utilisation en théorie des
  semigroupes}, Lectures on {Probability} {Theory}: {Ecole} d'{Eté} de
  {Probabilités} de {Saint}-{Flour} {XXII}-1992 (P.~Bernard, ed.), Lecture
  {Notes} in {Mathematics}, vol. 1581, Springer, Berlin, Heidelberg, 1994,
  pp.~1--114.

\bibitem{BGL-14}
D.~Bakry, I.~Gentil, and M.~Ledoux, \emph{Analysis and {G}eometry of {M}arkov
  {D}iffusion {O}perators}, A {S}eries of {C}omprehensive {S}tudies in
  {M}athematics, vol. 348, Springer, Cham, 2014.

\bibitem{BMP-09}
G.~Barles, S.~Mirrahimi, and B.~Perthame, \emph{Concentration in
  {Lotka}-{Volterra} parabolic or integral equations: a general convergence
  result}, Methods and Applications of Analysis \textbf{16} (2009), no.~3,
  321--340.

\bibitem{BEV-17}
N.~H. Barton, A.~M. Etheridge, and A.~V{\'e}ber, \emph{The infinitesimal model:
  {D}efinition, derivation, and implications}, Theoret. Population Biol.
  \textbf{118} (2017), 50--73.

\bibitem{BCV-16}
H.~Berestycki, J.~Coville, and H.-H. Vo, \emph{Persistence criteria for
  populations with non-local dispersion}, J. Math. Biol. \textbf{72} (2016),
  1693--1745.

\bibitem{BL-76}
H.~J. Brascamp and E.~H. Lieb, \emph{On extensions of the {B}runn-{M}inkowski
  and {P}r{\'e}kopa-{L}eindler theorems, including inequalities for log concave
  functions, and with an application to the diffusion equation}, J. Funct.
  Anal. \textbf{22} (1991), 366--389.

\bibitem{B-91}
Y~Brenier, \emph{Polar factorization and monotone rearrangement of
  vector-valued functions}, Comm. Pure Appl. Math. \textbf{44} (1991), no.~4,
  375--417.

\bibitem{C-92-2}
L.~A. Caffarelli, \emph{Boundary regularity of maps with convex potentials},
  Comm. Pure Appl. Math. \textbf{45} (1992), no.~9, 1141--1151.

\bibitem{C-92-1}
\bysame, \emph{The regularity of mappings with a convex potential}, J. Amer.
  Math. Soc. \textbf{5} (1992), no.~1, 99--104.

\bibitem{C-96}
\bysame, \emph{Boundary regularity of maps with convex potentials {II}}, Comm.
  Pure Appl. Math. \textbf{45} (1996), no.~9, 1141--1151.

\bibitem{C-00}
\bysame, \emph{Monotonicity properties of optimal transportation and the {FKG}
  and related inequalities}, Comm. Math. Phys. \textbf{214} (2000), no.~3,
  547--563.

\bibitem{C-02}
\bysame, \emph{Erratum: ``{M}onotonicity properties of optimal transportation
  and the {FKG} and related inequalities'' [{C}omm. {M}ath. {P}hys. 214 (2000),
  no. 3, 547-563]}, Comm. Math. Phys. \textbf{225} (2002), no.~2, 449--450.

\bibitem{CGP-19}
V.~Calvez, J.~Garnier, and F.~Patout, \emph{Asymptotic analysis of a
  quantitative genetics model with nonlinear integral operator}, J. \'Ec.
  polytech. Math. \textbf{6} (2019), 537--579.

\bibitem{CLP-21-arxiv}
V.~Calvez, L.~Lepoutre, and D.~Poyato, \emph{Ergodicity of the {F}isher
  infinitesimal model with quadratic selection}, 2021, arXiv:2107.00383.

\bibitem{CF-21}
M.~Colombo and M.~Fathi, \emph{Bounds on optimal transport maps onto
  log-concave measures}, J. Differ. Equ. \textbf{271} (2021), 1007--1022.

\bibitem{CFJ-17}
M.~Colombo, A.~Figalli, and Y.~Jhaveri, \emph{Lipschitz changes of variables
  between perturbations of log-concave measures}, Ann. Sc. Norm. Super. Pisa
  Cl. Sci. \textbf{17} (2017), 1491--1519.

\bibitem{CLW-17}
J.~Coville, F.~Li, and X.~Wang, \emph{On eigenvalue problems arising from
  nonlocal diffusion models}, Discrete Contin. Dyn. Syst. Ser. A \textbf{37}
  (2017), no.~2, 879--903.

\bibitem{DJMP-05}
O.~Diekmann, P.-E. Jabin, S.~Mischler, and B.~Perthame, \emph{The dynamics of
  adaptation: {A}n illuminating example and a {H}amilton–{J}acobi approach},
  Theor. Popul. Biol. \textbf{67} (2005), no.~4, 257--271.

\bibitem{FS-21}
V.~Ferrari and F.~Santambrogio, \emph{Lipschitz estimates on the {JKO} scheme
  for the {F}okker--{P}lanck equation on bounded convex domains}, Appl. Math.
  Lett. \textbf{112} (2021), 106806.

\bibitem{F-18}
R.~A. Fisher, \emph{The correlation between relatives on the supposition of
  mendelian inheritance}, Trans. Roy. Soc. Edinburgh \textbf{52} (1918),
  399--433.

\bibitem{F-22}
\bysame, \emph{On the mathematical foundations of theoretical statistics},
  Philos. Trans. Royal Soc. A \textbf{222} (1922), no.~594-604, 309--368.

\bibitem{GCBBLRC-22-arxiv}
J.~Garnier, O.~Cotto, T.~Bourgeron, E.~Bouin, T.~Lepoutre, O.~Ronce, and
  V.~Calvez, \emph{Adaptation to a changing environment: what me normal?},
  2022, arXiv:2206.13248.

\bibitem{J-19}
Y.~Jhaveri, \emph{On the (in)stability of the identity map in optimal
  transportation}, Calc. Var. Partial Differ. Equ. \textbf{58} (2019), 96.

\bibitem{M-07}
R.~Mahadevan, \emph{A note on a non-linear {K}rein-{R}utman theorem}, Nonlinear
  Anal. Theory Methods Appl. \textbf{67} (2007), no.~11, 3084--3090.

\bibitem{MR-13}
S.~Mirrahimi and G.~Raoul, \emph{Dynamics of sexual populations structured by a
  space variable and a phenotypical trait}, Theoret. Population Biol.
  \textbf{84} (2013), 87--103.

\bibitem{N-88}
R.~D. Nussbaum, \emph{Hilbert's {Projective} {Metric} and {Iterated}
  {Nonlinear} {Maps}}, American Mathematical Society, Basel, 1099.

\bibitem{N-94}
\bysame, \emph{Finsler structures for the part metric and {Hilbert}'s
  projective metric and applications to ordinary differential equations},
  Differ. Integral Equ. \textbf{7} (1994), no.~5-6, 1649--1707.

\bibitem{P-20-arxiv}
F.~Patout, \emph{The {C}auchy problem for the infinitesimal model in the regime
  of small variance}, 2020, arXiv:2001.04682.

\bibitem{PB-08}
B.~Perthame and G.~Barles, \emph{Dirac concentrations in {Lotka}-{Volterra}
  parabolic {PDEs}}, Indiana Univ. Math. J. \textbf{57} (2008), no.~7,
  3275--3301.

\bibitem{R-17-arxiv}
G.~Raoul, \emph{Macroscopic limit from a structured population model to the
  {K}irkpatrick-{B}arton model}, 2017, arXiv:1706.04094.

\bibitem{SW-14}
A.~Saumard and J.~A. Wellner, \emph{Log-concavity and strong log-concavity: {A}
  review}, Statist. Surv. \textbf{8} (2014), no.~45, 45--114.

\bibitem{S-05}
S.~Stigler, \emph{Fisher in 1921}, Stat. Sci. \textbf{20} (2005), no.~1,
  32--49.

\bibitem{V-03}
C.~Villani, \emph{Topics in optimal transportatio}, American Mathematical
  Society, Providence, RI, 2003.

\end{thebibliography}

\end{document}